\documentclass[12pt,english]{article}

\usepackage[T1]{fontenc}
\usepackage[latin9]{inputenc}
\usepackage{geometry}
\usepackage{amsmath}
\usepackage{amsthm}
\usepackage{amssymb}
\usepackage{stackrel}
\usepackage{setspace}
\usepackage[authoryear]{natbib}
\usepackage{xcolor}
\usepackage{bm}
\usepackage{graphicx}
\usepackage{comment}

\usepackage[colorlinks,linkcolor=black,citecolor=blue, pagebackref=true]{hyperref}
\renewcommand*{\backrefalt}[4]{%
    \ifcase #1 \footnotesize{(Not cited.)}%
    \or        \footnotesize{(Cited on page~#2.)}%
    \else      \footnotesize{(Cited on pages~#2.)}%
    \fi}
\usepackage{bookmark}
\providecommand{\keywords}[1]
{
	\small	
	\textbf{\textit{Keywords.}} #1
}

\usepackage{authblk}

\makeatletter
\theoremstyle{plain}
\newtheorem{prop}{\protect\propositionname}
\theoremstyle{plain}
\newtheorem{thm}{\protect\theoremname}
\theoremstyle{plain}
\newtheorem{cor}{\protect\corollaryname}
\theoremstyle{plain}
\newtheorem{lem}{\protect\lemmaname}
\theoremstyle{remark}
\newtheorem{rem}{\protect\remarkname}
\makeatother

\usepackage[british]{babel}
\providecommand{\corollaryname}{Corollary}
\providecommand{\propositionname}{Proposition}
\providecommand{\theoremname}{Theorem}
\providecommand{\lemmaname}{Lemma}
\providecommand{\remarkname}{Remark}

\begin{document}
\title{On the large-sample limits of some Bayesian model evaluation statistics}
\author[1,2]{Hien Duy Nguyen}
\author[3]{Mayetri Gupta}
\author[4]{Jacob Westerhout}
\author[5]{TrungTin Nguyen}

\affil[1]{Department of Mathematics and Physical Science, La Trobe University, Melbourne, Australia.}
\affil[2]{Institute of Mathematics for Industry, Kyushu University, Fukuoka, Japan.}
\affil[3]{School of Mathematics and Statistics, University of Glasgow, Glasgow, UK.}
\affil[4]{School of Mathematics and Physics, University of Queensland, Brisbane, Australia.}
\affil[5]{School of Mathematical Sciences, Queensland University of Technology, Brisbane, Australia.}

\maketitle
\begin{abstract}
Model selection and order selection problems frequently arise in statistical
practice. A popular approach to addressing these problems in the frequentist
setting involves information criteria based on penalised maxima of
log-likelihoods for competing models. In the Bayesian context, similar
criteria are employed, replacing the maximised log-likelihoods with
posterior expectations of the log-likelihood. Despite their popularity in applications,
the large-sample behaviour of these criteria---such as the deviance
information criterion (DIC), Bayesian predictive information criterion
(BPIC), and widely applicable Bayesian information criterion (WBIC)---has
received relatively little attention. In this work, we investigate
the almost-sure limits of these criteria and establish novel results
on posterior and generalised posterior consistency, which are of independent
interest. The utility of our theoretical findings is demonstrated
via illustrative technical and numerical examples.

\end{abstract}



\section{Introduction\label{sec:introduction}}

\subsection{General introduction}

Bayesian statistics provides a cohesive and comprehensive set of tools
for the analysis of data arising from natural and man-made phenomena.
In recent decades, with advances in the availability of computational
resources, Bayesian analysis has become an increasingly popular paradigm
for the development of novel statistical methods in applied settings,
and Bayesian statistical inference is an increasingly active area of
fundamental research and development \citep{ghosh2006introduction, robert2007bayesian}. 

As in frequentist settings, the Bayesian framework often requires
the analyst to choose between various models that may each provide
plausible fits to the observed data. To this end, a popular technique
is to employ information criteria, such as the
Akaike information criterion (AIC; \citealp{Akaike1974}) and the Bayesian
information criterion (BIC; \citealp{Schwarz1978}), in order to evaluate
the goodness of fit of competing models. Both the AIC and BIC are
defined via the maximised log-likelihood function of the data.
Some texts that explore these information criteria in detail are
\citet{Burnham2002} and \citet{Konishi2008}. The limiting behaviours
of the AIC, BIC, and other information criteria based on the maximum
of a data-dependent objective have been investigated in general
contexts; see, for example, \citet{Sin:1996aa} and \citet{Baudry:2015aa}.

As an alternative to the maximum-likelihood approach, some authors
have suggested information criteria based on the posterior expectation
of the log-likelihood function, as well as the log-likelihood function
evaluated at Bayesian point summaries, such as the posterior mean.
Comprehensive resources for these types of
criteria are \citet{Ando2010Bayesian-Model-} and \citet{watanabe2018mathematical},
which discuss examples including the deviance information
criterion (DIC; \citealp{spiegelhalter2002bayesian}), the Bayesian
predictive information criterion (BPIC; \citealp{ando2007bayesian}),
and the widely applicable Bayesian information criterion (WBIC; \citealp{watanabe2013widely}).

In this work, we seek to provide general regularity conditions under
which the almost-sure limits of the DIC, BPIC, and WBIC can be derived. Our technique relies on the concept
of almost-sure weak convergence, otherwise known as almost-sure conditional
weak convergence, as espoused in the works of \citet{sweeting1989conditional},
\citet{berti2006almost}, and \citet{grubel2016functional}. This
is the mode of convergence under which Bayesian consistency is often
established, and under which limit theorems for the bootstrap are proved; see, for
example, \citet[Ch.~4]{Ghosh:2003aa}, \citet[Ch.~6]{ghosal2017fundamentals},
and \citet[Sec.~1.13 and~3.7]{van-der-Vaart:2023aa}. The almost-sure
convergence of posterior measures is then combined with the convergence
of averaged functions of data and parameters using generalisations
of the Lebesgue convergence theorems, which allow for varying sequences of measures and integrands, as described in the works of
\citet{serfozo1982convergence}, \citet{feinberg2020fatou}, and \citet{feinberg2020fatoub}.

We note that convergence in mean value for the listed information
criteria and related objects can be obtained, for example, via the
many results of \citet{watanabe2018mathematical}. Our almost-sure
convergence results can thus be seen as complementary to the approach
of \citet{watanabe2018mathematical}.
We further note that our results complement those of \cite{li2025dic_theory}, who recently provided an assessment of the DIC as a predictive risk criterion at the expectation level, via a Laplace-approximation approach under strong regularity, which contrasts with our minimal and general approach to the DIC and other criteria.
We shall provide technical details in the sequel.

\subsection{Technical introduction}

Let $\left(\Omega,\mathfrak{F},\text{P}\right)$ be a probability
space with typical element $\omega$ and expectation operator $\text{E}$,
and let $\left(\mathbb{T},\mathfrak{B}_{\mathbb{T}}\right)$ be a
measurable space with typical element $\theta$, where $\mathbb{T}$ is the intersection of an open subset and a closed subset of $\mathbb{R}^{p}$ with its Borel $\sigma$-algebra $\mathfrak{B}_{\mathbb{T}}$.
For each $n\in\mathbb{N}$, let $f_{n}:\Omega\times\mathbb{T}\rightarrow\mathbb{R}$
be such that, for each $\theta\in\mathbb{T}$, the mapping $\omega\mapsto f_{n}\!\left(\omega,\theta\right)$ is $\mathfrak{F}$-measurable,
and let $\text{Q}_{n}:\Omega\times\mathfrak{B}_{\mathbb{T}}\rightarrow\mathbb{R}_{\ge0}$
be a random measure in the sense that, for each $\omega$, $\text{Q}_{n}\!\left(\omega,\cdot\right)\in\mathcal{P}\!\left(\mathbb{T}\right)$,
where $\mathcal{P}\!\left(\mathbb{T}\right)$ is the space of measures
on $\left(\mathbb{T},\mathfrak{B}_{\mathbb{T}}\right)$.

Following the exposition of \citet{grubel2016functional}, we say
that the sequence $\left(\text{Q}_{n}\right)_{n\in\mathbb{N}}$
converges weakly to $\text{Q}\in\mathcal{P}\left(\mathbb{T}\right)$,
almost surely, with respect to the measure $\text{P}$ (which we denote
as $\text{Q}_{n}\stackrel[n\rightarrow\infty]{\text{P}\text{-a.s.w.}}{\longrightarrow}\text{Q}$),
if for each bounded and continuous $f:\mathbb{T}\rightarrow\mathbb{R}$,
\begin{equation} \label{eq:-weak-con-def}
\int_{\mathbb{T}}f\left(\theta\right)\text{Q}_{n}\left(\omega,\text{d}\theta\right)\underset{n\rightarrow\infty}{\longrightarrow}\int_{\mathbb{T}}f\left(\theta\right)\text{Q}\left(\text{d}\theta\right)\text{,}
\end{equation}
for every $\omega$ on a $\text{P}$-almost sure ($\text{P}$-a.s.) set (where the set may depend on $f$). Now suppose that the sequence of random functions $\left(f_{n}\left(\omega,\cdot\right)\right)_{n\in\mathbb{N}}$
converges to some deterministic $\left(\mathbb{T},\mathfrak{B}_{\mathbb{T}}\right)$-measurable
function $f$, in some appropriate sense, on a $\text{P}$-a.s. set. Then, we
seek to investigate integrals of the form
\begin{equation}
\int_{\mathbb{T}}f_{n}\left(\omega,\theta\right)\text{Q}_{n}\left(\omega,\text{d}\theta\right)\text{,}\label{eq:-general-integral-expression}
\end{equation}
and deduce sufficient conditions under which such integrals will converge
to the limit
\[
\int_{\mathbb{T}}f\left(\theta\right)\text{Q}\left(\text{d}\theta\right)\text{,}
\]
for each $\omega$ on a $\text{P}$-a.s. set.

We note that such integrals of the form \eqref{eq:-general-integral-expression}
occur frequently in Bayesian statistics. To present some common cases,
let us introduce some further notation. Let $\left(X_{i}\right)_{i\in\left[n\right]}$
be a sequence of random variables $X_{i}:\Omega\rightarrow\mathbb{X}$
for each $i\in\left[n\right]$, where
$\mathbb{X}\subset\mathbb{R}^{q}$ is a subset with $\sigma$-algebra
$\mathfrak{B}_{\mathbb{X}}$. Next, we suppose that $\Pi\in\mathcal{P}\left(\mathbb{T}\right)$
is a prior probability measure on $\left(\mathbb{T},\mathfrak{B}_{\mathbb{T}}\right)$,
with density function $\pi:\mathbb{T}\rightarrow\mathbb{R}_{\ge0}$
with respect to some dominating measure $\mathfrak{m}$ on $\left(\mathbb{T},\mathfrak{B}_{\mathbb{T}}\right)$. 

Next, we suppose that $\text{p}\left(x_{1},\dots,x_{n}\mid\theta\right)$
is a likelihood function in the sense that, for each fixed $\theta\in\mathbb{T}$,
$\text{p}\left(\cdot\mid\theta\right):\mathbb{X}^{n}\rightarrow\mathbb{R}_{\ge0}$
is a probability density function with respect to the $n$-fold product
measure of some dominating measure $\mathfrak{n}$ on $\left(\mathbb{X},\mathfrak{B}_{\mathbb{X}}\right)$.
Then, we may identify the sequence of posterior measures $\left(\Pi_{n}\left(\omega,\cdot\right)\right)_{n\in\mathbb{N}}$
as a sequence of random measures, where
\begin{equation}
\Pi_{n}\left(\omega,\cdot\right)=\Pi\left(\cdot\mid X_{1}(\omega),\dots,X_{n}(\omega)\right)\text{,}\label{eq:-posterior-measure}
\end{equation}
is defined by its density function
\[
\theta\mapsto\pi\left(\theta\mid X_{1},\dots,X_{n}\right)=\frac{\text{p}\left(X_{1},\dots,X_{n}\mid\theta\right)\pi\left(\theta\right)}{\int_{\mathbb{T}}\text{p}\left(X_{1},\dots,X_{n}\mid\tau\right)\pi\left(\tau\right)\mathfrak{m}\left(\text{d}\tau\right)}\text{.}
\]

Suppose further that $\left(\theta_{n}\right)_{n\in\mathbb{N}}$ is
a sequence of estimators of $\theta$, in the sense that $\theta_{n}:\mathbb{X}^{n}\rightarrow\mathbb{T}$
for each $n\in\mathbb{N}$. A first example of an object of the form
\eqref{eq:-general-integral-expression} is to consider parametric
loss functions of the form $\ell:\mathbb{T}\times\mathbb{T}\rightarrow\mathbb{R}_{\ge0}$.
Then, we may identify $\omega\mapsto\ell\left(\theta_{n}(\omega),\theta\right)$ as a random
function and write the posterior risk with respect to $\ell$ as
defined in \citet[Ch.~4]{Lehmann1998} by the expression 
\[
\int_{\mathbb{T}}\ell\left(\theta_{n}\left(\omega\right),\theta\right)\Pi_{n}\left(\omega,\text{d}\theta\right)\text{.}
\]
More pertinent to our investigation, we can consider more elaborate
random functions $f_{n}$ that are best described as utility or score
functions $u_{n}:\mathbb{X}^{n}\times\mathbb{T}\rightarrow\mathbb{R}$,
as per the exposition of \citet{Berger:1985aa} and \citet{bernardo2009bayesian}.
Examples of such utility functions include the log-likelihood function
\[
u_{n}\left(x_{1},\dots,x_{n};\theta\right)=\log\text{p}\left(x_{1},\dots,x_{n}\mid\theta\right)\text{.}
\]

Here, the integral of the log-likelihood with respect to the posterior measure is of particular interest as it appears
in both the DIC and BPIC, while the integral with respect to a modified posterior measure, described below, forms the definition of the WBIC.

We can also identify other Bayesian objects
with sequences of measures $\left(\text{Q}_{n}\left(\omega,\cdot\right)\right)_{n\in\mathbb{N}}$.
For example, approximate Bayesian computation pseudo-posteriors can
be identified as random measures (see, e.g., \citealp{bernton2019approximate}, \citealp{nguyen2020approximate} and \citealp{forbes_summary_2022}), as with asymptotic Bayesian
approximations of the posterior distribution obtained by the Bernstein--von
Mises theorem (see, e.g., \citealp[Ch. 10]{vdVaart1998} and \citealp[Ch. 4]{ghosh2006introduction}).
Related to such objects are data-dependent measures on $\left(\mathbb{T},\mathfrak{B}_{\mathbb{T}}\right)$
often studied as fiducial distributions or confidence distributions
(see, e.g., \citealp{hannig2016generalized} and \citealp{schweder2016confidence}).
A general method for generating sequences $\left(\text{Q}_{n}\right)_{n\in\mathbb{N}}$
that exhibit $\text{P}$-a.s. convergence, including variational Bayesian posteriors (see, e.g., \citealp{zhang_convergence_2020} and \citealp{nguyen_bayesian_2024}),
generalised posteriors (see, e.g., \citealp{Zhang:2006ab} and
\citealp{Bissiri:2016aa}), and typical posteriors of the form \eqref{eq:-posterior-measure},
has recently been studied by \citet{Knoblauch:2022aa} (see also
\citealp{knoblauch2019frequentist}). Of particular interest in our
study is the generalised (power) posterior $\Pi_{n}^{\beta_{n}}$,
for a sequence $\left(\beta_{n}\right)_{n\in\mathbb{N}}\subset\mathbb{R}_{>0}$,
whose density is defined by
\[
\pi^{\beta_{n}}\left(\theta\mid X_{1},\dots,X_{n}\right)=\frac{\left\{\text{p}\left(X_{1},\dots,X_{n}\mid\theta\right)\right\}^{\beta_{n}}\pi\left(\theta\right)}{\int_{\mathbb{T}}\left\{\text{p}\left(X_{1},\dots,X_{n}\mid\tau\right)\right\}^{\beta_{n}}\pi\left(\tau\right)\mathfrak{m}\left(\text{d}\tau\right)}\text{,}
\]
which appears in the definition of the WBIC (in particular, $\beta_{n}=1/\log n$).

The modes of weak convergence, almost surely and in probability, have
previously been formally studied in the works of \citet{sweeting1989conditional},
\citet{berti2006almost}, \citet{bain2009fundamentals}, and \citet{grubel2016functional}. Let $\delta_{\theta_{0}}:2^\mathbb{T}\rightarrow\left\{0,1\right\}$
denote the delta measure that takes the value $1$ if $\theta_{0}\in\mathbb{B}$
and $0$ otherwise, for every $\mathbb{B}\in2^\mathbb{T}$.
Then, the usual notion of posterior consistency, with respect to a
parameter $\theta_{0}\in\mathbb{T}$, in Bayesian analysis, can be
characterised as the weak convergence
of the posterior measure $\Pi_n^1=\Pi_n$ to $\delta_{\theta_{0}}$,
$\text{P}$-a.s. (cf. \citealp[Prop. 6.2]{ghosal2017fundamentals}). For
another example, if the sequence $\left(\theta_{n}\right)_{n\in\mathbb{N}}$
converges $\text{P}$-a.s. to $\theta_{0}$, then since $\mathbb{T}$ is separable, we can determine weak convergence by checking condition \eqref{eq:-weak-con-def} using only a countable number of bounded continuous functions (cf. \citealp[Thm. 2.18]{bain2009fundamentals}), and determine that
$\left(\delta_{\theta_{n}}\right)_{n\in\mathbb{N}}$
converges to $\delta_{\theta_{0}}$, $\text{P}$-a.s.w., by the continuous mapping theorem. Since we can write
\begin{equation}
f_{n}\left(\omega,\theta_{n}\right)=\int_{\mathbb{T}}f_{n}\left(\omega,\theta\right)\delta_{\theta_{n}}\left(\text{d}\theta\right)\text{, and }f\left(\theta_{0}\right)=\int_{\mathbb{T}}f\left(\theta\right)\delta_{\theta_{0}}\left(\text{d}\theta\right)\text{,}\label{eq:-delta-measures}
\end{equation}
the convergence of $\left(f_{n}\left(\cdot,\theta_{n}\right)\right)_{n\in\mathbb{N}}$
to $f\left(\theta_{0}\right)$, $\text{P}$-a.s., falls within the context of
this work.

The remainder of the manuscript is organised as follows. Section \ref{section_main_result} presents the main results; Section \ref{section_examples} provides illustrative applications with numerical simulations; and Section \ref{section_discussion} offers a concluding discussion. Appendix~\ref{section_proofs} contains proofs and technical details, and additional technical calculations for the examples are collected in Appendix \ref{sec: Appendix-B}.

\section{Main results}\label{section_main_result}

In the sequel, we will retain the notation and technical setup of
the introduction. Furthermore, to simplify discussions, we will assume
that $\left(X_{i}\right)_{i\in\mathbb{N}}$ is an independent and identically distributed (IID) sequence of
random variables, each identical to $X:\Omega\rightarrow\mathbb{X}$.

We are largely concerned with three information criteria, which have
found particular popularity in the Bayesian literature; namely, the
DIC of \citet{spiegelhalter2002bayesian}, the BPIC of \citet{Ando2010Bayesian-Model-},
and the WBIC of \citet{watanabe2013widely}, taking the respective
forms:
\begin{equation}
\text{DIC}_{n}=-\frac{4}{n}\int_{\mathbb{T}}\sum_{i=1}^{n}\log\text{p}\left(X_{i}\mid\theta\right)\Pi_{n}\left(\text{d}\theta\right)+\frac{2}{n}\sum_{i=1}^{n}\log\text{p}\left(X_{i}\mid\bar{\theta}_{n}\right)\text{,}\label{eq:-dic}
\end{equation}
\begin{align}
\text{BPIC}_{n}=-\frac{2}{n}\int_{\mathbb{T}}\sum_{i=1}^{n}\log\text{p}\left(X_{i}\mid\theta\right)\Pi_{n}\left(\text{d}\theta\right)+2\frac{p}{n}\text{,} \label{eq:-bpic}
\end{align}
and
\[
\text{WBIC}_{n}=-\frac{2}{n}\int_{\mathbb{T}}\sum_{i=1}^{n}\log\text{p}\left(X_{i}\mid\theta\right)\Pi_{n}^{1/\log n}\left(\text{d}\theta\right)\text{,}
\]
where
\[
\bar{\theta}_{n}=\left(\int_{\mathbb{T}}\theta_{1}\Pi_{n}\left(\text{d}\theta\right),\dots,\int_{\mathbb{T}}\theta_{j}\Pi_{n}\left(\text{d}\theta\right),\dots,\int_{\mathbb{T}}\theta_{p}\Pi_{n}\left(\text{d}\theta\right)\right)
\]
is the posterior mean, with $\theta_j:\mathbb{R}^p\to\mathbb{R}$ denoting the $j$th coordinate projection of $\theta$. Here, $x\mapsto\text{p}\left(x\mid\theta\right)$
characterises a family of probability density functions (PDFs), indexed
by $\theta\in\mathbb{T}$, with
\[
\text{p}\left(x_{1},\dots,x_{n}\mid\theta\right)=\prod_{i=1}^{n}\text{p}\left(x_{i}\mid\theta\right)
\]
in the definitions of $\Pi_{n}$ and $\Pi_{n}^{\beta_{n}}$. 

Observe that the first expression in each of the criteria is of the form
\eqref{eq:-general-integral-expression} with
\begin{equation}
f_{n}\left(\omega,\theta\right)=-\frac{2}{n}\sum_{i=1}^{n}\log\text{p}\left(X_{i}\mid\theta\right)\text{,}\label{eq:-generic-average-log-like}
\end{equation}
in each case, and where $\text{Q}_{n}$
equals $\Pi_{n}^{\beta_{n}}$ for the WBIC, for some choice of $\left(\beta_{n}\right)_{n\in\mathbb{N}}$.
In particular, the BPIC and DIC take $\beta_{n}=1$, whereas $\beta_{n}=1/\log n$
is used in the original definition of the WBIC. The second term in
the DIC also has the form \eqref{eq:-general-integral-expression}, where
$f_{n}$ has the form \eqref{eq:-generic-average-log-like} but $\text{Q}_{n}$
equals $\delta_{\bar{\theta}_{n}}$.

Note that we have presented the information criteria as functions
of sample averages for convenience. We have also scaled the WBIC by a factor of two to standardise the scales of the three quantities.
That is, in our notation, the original characterisations of the DIC, BPIC,
and WBIC present the objects in the forms: $n\text{DIC}_{n}$, $n\text{BPIC}_{n}$,
and $n\text{WBIC}_{n}/2$, respectively.

\subsection{Consistency of \texorpdfstring{$\Pi_{n}^{\beta_{n}}$}{Pinbetan}}

Let $\mathbb{B}_{\epsilon}\left(\theta\right)=\left\{ \tau\in\mathbb{T}:\left\Vert \theta-\tau\right\Vert <\epsilon\right\}$
and $\bar{\mathbb{B}}_{\epsilon}\left(\theta\right)=\left\{ \tau\in\mathbb{T}:\left\Vert \theta-\tau\right\Vert \le\epsilon\right\}$
denote the open and closed balls around $\theta\in\mathbb{T}$ of
radius $\epsilon>0$, respectively. We shall also use $\left(\cdot\right)^{\text{c}}$
to denote the set complement with respect to the universe $\mathbb{T}$. Make the following additional assumptions:
\begin{description}
\item [{A1}] The PDF $\text{p}$ is Carath\'eodory
in the sense that $\text{p}\left(\cdot\mid\theta\right)$ is measurable
for each $\theta\in\mathbb{T}$, $\text{p}\left(x\mid\cdot\right)$
is continuous for each $x\in\mathbb{X}$, and $\text{p}\left(x\mid\theta\right)>0$
for each $\theta\in\mathbb{T}$ and $x\in\mathbb{X}$.
\item [{A2}] There exists a $\text{P}$-a.s. set on which, for every compact
$\mathbb{K}\subset\mathbb{T}$,
\begin{equation}
    \sup_{\theta\in\mathbb{K}}\left|\frac{1}{n}\sum_{i=1}^{n}\log\text{p}\left(X_{i}\mid\theta\right)-\text{E}\!\left[\log\text{p}\left(X\mid\theta\right)\right]\right|\underset{n\rightarrow\infty}{\longrightarrow}0\text{.}\label{eq:-uniform-convergence-loglike}
\end{equation}

\item [{A3}] There exists a unique maximiser $\theta_{0}\in\mathbb{T}$
of $\text{E}\!\left[\log\text{p}\left(X\mid\theta\right)\right]$, in the sense that, for each $\theta\in\mathbb{T}\backslash\left\{ \theta_{0}\right\}$, 
\[
\text{E}\!\left[\log\text{p}\left(X\mid\theta\right)\right]<\text{E}\!\left[\log\text{p}\left(X\mid\theta_{0}\right)\right]\text{.}
\]
\item [{A4}] There exists a compact set $\mathbb{S}\subsetneq\mathbb{T}$ containing $\theta_0$,
such that
there is a $\text{P}$-a.s. set on which
\[
\limsup_{n\rightarrow\infty}\sup_{\theta\in\mathbb{S}^{\text{c}}}\frac{1}{n}\sum_{i=1}^{n}\log\text{p}\left(X_{i}\mid\theta\right)<\text{E}\!\left[\log\text{p}\left(X\mid\theta_{0}\right)\right]\text{.}
\]
\end{description}
\begin{prop}
\label{prop:-consistency-of-beta-posterior}Assume that $\left(X_{i}\right)_{i\in\mathbb{N}}$
are IID, $\Pi\left(\mathbb{B}_{\rho}\left(\theta_{0}\right)\right)>0$
for every $\rho>0$, and A1--A4 hold. If $n\beta_{n}\rightarrow\infty$,
then the sequence of posterior measures $\left(\Pi_{n}^{\beta_{n}}\right)_{n\in\mathbb{N}}$
is consistent with respect to $\theta_{0}\in\mathbb{T}$ in the sense
that for every $\epsilon>0$,
\begin{equation}
\Pi_{n}^{\beta_{n}}\left(\bar{\mathbb{B}}_{\epsilon}\left(\theta_{0}\right)\right)\xrightarrow[n\rightarrow\infty]{\text{P}\text{-a.s.}}1\text{.}\label{eq:-consistency-of-beta-posterior}
\end{equation}
\end{prop}
We shall establish a more general result in the next section, which implies Proposition \ref{prop:-consistency-of-beta-posterior}. The proof of this proposition is provided in Appendix~
\ref{proof_prop_consistency_beta_posterior}. Assumption A2 provides an almost-sure event that guarantees
the uniform convergence of the average log-likelihood on all compact
sets, which we then use to guarantee continuous convergence of
the average log-likelihood over the entirety of $\mathbb{T}$, since
it is the intersection of an open and closed set, it is a locally compact separable metric space, and is hence hemicompact (see \citealp[Def. 2.1]{beckenstein2011topological}). A3 is
an identifiability assumption regarding the maximiser $\theta_{0}$
of $\text{E}\!\left[\log\text{p}\left(X\mid\theta\right)\right]$, requiring
that $\text{E}\!\left[\log\text{p}\left(X\mid\theta\right)\right]$ be uniquely maximised on $\mathbb{T}$.
Taken together with A1 and A3, A4 can also be verified if the uniform convergence on compact sets $\mathbb{K}\subset\mathbb{T}$ in A2 is instead replaced by uniform convergence over the entirety of $\mathbb{T}$, $\text{P}$-a.s. By this observation, under A1--A3, A4 is immediately satisfied if $\mathbb{T}$ is compact. In the sequel, we further detail a useful approach, based on quasiconvexity of the log-likelihood functions.

Based on \citet[Prop. 6.2]{ghosal2017fundamentals}, we have the fact
that \eqref{eq:-consistency-of-beta-posterior} is equivalent to the
statement that $\left(\Pi_{n}^{\beta_{n}}\right)_{n\in\mathbb{N}}$
converges to $\delta_{\theta_{0}}$, $\text{P}$-a.s.w. We consider that both
A1 and A3 are fundamental assumptions about the properties of the PDFs
$\text{p}$ and must be inspected for each problem. However, A2 can be
implied by more verifiable assumptions using a uniform strong law
of large numbers (e.g., Theorem \ref{thm:-ulln}). Make the following
additional assumptions:
\begin{description}
\item [{A2a}] For each $\theta\in\mathbb{T}$, 
\[
\frac{1}{n}\sum_{i=1}^{n}\log\text{p}\left(X_{i}\mid\theta\right)\xrightarrow[n\rightarrow\infty]{\text{P}\text{-a.s.}}\text{E}\!\left[\log\text{p}\left(X\mid\theta\right)\right]\text{,}
\]
where $\text{E}\!\left[\left|\log\text{p}\left(X\mid\theta\right)\right|\right]<\infty$.
\item [{A2b}] For each compact set $\mathbb{K}\subset\mathbb{T}$, there
exists a dominating function $\Delta:\mathbb{X}\rightarrow\mathbb{R}$,
such that
\[
\left|\log\text{p}\left(X\mid\theta\right)\right|\le\Delta\left(X\right)\text{,}
\]
for each $\theta\in\mathbb{K}$, and $\text{E}\!\left[\Delta\left(X\right)\right]<\infty$.
\end{description}
The conclusion of Proposition \ref{prop:-consistency-of-beta-posterior}
holds with A2 replaced by A2a and A2b, above. Thus, we now have enough practical
conditions to establish that $\left(\Pi_{n}^{\beta_{n}}\right)_{n\in\mathbb{N}}$
converges to $\delta_{\theta_{0}}$, $\text{P}$-a.s.w.

\subsection{A general convergence theorem}

We now present conditions under which \eqref{eq:-general-integral-expression}
converges to its limit, generally. To this end, we adapt the deterministic
generalised Lebesgue convergence theorems of \citet{serfozo1982convergence},
\citet{feinberg2020fatoub} and \citet{feinberg2020fatou}. 

To proceed, we require a number of definitions. First, we will say
that a sequence of random functions $\left(f_{n}(\omega,\cdot)\right)_{n\in\mathbb{N}}$
continuously converges to the deterministic function $f:\mathbb{T}\rightarrow\mathbb{R}$,
$\mathrm{P}$-a.s., if for every sequence $\left(\theta_{n}\right)_{n\in\mathbb{N}}$
converging to a limit $\theta\in\mathbb{T}$, it holds that $\left(f_{n}\left(\omega,\theta_{n}\right)\right)_{n\in\mathbb{N}}$
converges to $f\left(\theta\right)$, for every $\omega$ on
a $\mathrm{P}$-a.s. set. Next, we say that $\left(f_{n}\right)_{n\in\mathbb{N}}$
is asymptotically uniformly integrable (AUI) with respect to $\left(\mathrm{Q}_{n}\right)_{n\in\mathbb{N}}$,
$\mathrm{P}$-a.s., if
\[
\lim_{\delta\rightarrow\infty}\limsup_{n\rightarrow\infty}\int_{\mathbb{T}}\left|f_{n}\left(\omega,\theta\right)\right|
\chi_{\{\tau\in\mathbb{T}:\,|f_{n}(\omega,\tau)|\ge\delta\}}(\theta)
\mathrm{Q}_n\left(\omega,\mathrm{d}\theta\right)=0\text{,}
\]
for every $\omega$ on a $\mathrm{P}$-a.s. set. If each $f_{n}=f$, then we will instead say that $f$ is AUI with respect to $\left(\mathrm{Q}_{n}\right)_{n\in\mathbb{N}}$. Here, $\chi_{\mathbb{A}}:\mathbb{T}\rightarrow\mathbb{R}$
is the usual characteristic function for the set $\mathbb{A}\subset\mathbb{T}$.
A direct modification of \citet[Cor. 2.8]{feinberg2020fatou} for
the stochastic setting yields the following general Lebesgue dominated
convergence theorem. 
\begin{thm} \label{thm:-feinberg}
Assume that $\left(\mathrm{Q}_{n}\right)_{n\in\mathbb{N}}$ converge
to $\mathrm{Q}$, $\mathrm{P}$-a.s.w., and $\left(f_{n}\right)_{n\in\mathbb{N}}$
continuously converge to $f$, $\mathrm{P}$-a.s. If $\left(f_{n}\right)_{n\in\mathbb{N}}$
is AUI with respect to $\left(\mathrm{Q}_{n}\right)_{n\in\mathbb{N}}$,
$\mathrm{P}$-a.s., then 
\[
\lim_{n\rightarrow\infty}\int_{\mathbb{T}}f_{n}\left(\omega,\theta\right)\mathrm{Q}_{n}\left(\omega,\mathrm{d}\theta\right)=\int_{\mathbb{T}}f\left(\theta\right)\mathrm{Q}\left(\mathrm{d}\theta\right)\text{,}
\]
for almost every $\omega$ with respect to $\mathrm{P}$.
\end{thm}
To verify that $\left(f_{n}\right)_{n\in\mathbb{N}}$ continuously
converges to $f$, $\mathrm{P}$-a.s., we can use the equivalent notion that $f$
is continuous on $\mathbb{T}$, and for almost every $\omega$ with
respect to $\mathrm{P}$, for every compact $\mathbb{K}\subset\mathbb{T}$,
$\left(f_{n}\left(\omega,\cdot\right)\right)_{n\in\mathbb{N}}$ converges
uniformly to $f$ on $\mathbb{K}$ (cf. \citealt[Thm. 3.1.5]{Remmert:1991aa}).
At first this might seem daunting since we would have to construct
an almost sure event of $\Omega$ on which uniform convergence holds
for an uncountable number of compact subsets $\mathbb{K}\subset\mathbb{T}$.
However, we shall use the fact that, since $\mathbb{T}\subset\mathbb{R}^{p}$ is the intersection of an open and closed set,
it is hemicompact, which then provides the necessary regularity
to construct such an almost sure event. We shall demonstrate applications
of these facts in the sequel. 

The AUI property is more difficult to verify but can be implied by
more practical assumptions. For instance, \citet[Cor. 2.10]{feinberg2020fatou}
suggests that if there exists a sequence of stochastic functions $\left(g_{n}\right)_{n\in\mathbb{N}}$,
converging continuously to $g:\mathbb{T}\rightarrow\mathbb{R}$, $\mathrm{P}$-a.s.,
with $g_{n}:\Omega\times\mathbb{T}\rightarrow\mathbb{R}$, for each
$n\in\mathbb{N}$, such that for almost every $\omega$ with respect
to $\mathrm{P}$, $\left|f_{n}\left(\omega,\theta\right)\right|\le g_{n}\left(\omega,\theta\right)$,
for every $\theta\in\mathbb{T}$, and 
\begin{equation}
\limsup_{n\rightarrow\infty}\int_{\mathbb{T}}g_{n}\left(\omega,\theta\right)\mathrm{Q}_{n}\left(\omega,\mathrm{d}\theta\right)\le\int_{\mathbb{T}}g\left(\theta\right)\mathrm{Q}\left(\mathrm{d}\theta\right)<\infty\text{,}\label{eq:-aui-sufficient-cond}
\end{equation}
then $\left(f_{n}\right)_{n\in\mathbb{N}}$ is AUI with respect to
$\left(\mathrm{Q}_{n}\right)_{n\in\mathbb{N}}$, $\mathrm{P}$-a.s., if $\left(f_{n}\right)_{n\in\mathbb{N}}$
continuously converges to $f$, $\mathrm{P}$-a.s. We note that a version of
this criterion, when $f_n=f$ for each $n$, is given in \citet[Thm. 1.11.3]{van-der-Vaart:2023aa}. Trivially, the condition holds if there exists a constant $b<\infty$, such that $\left|f_{n}\left(\omega,\theta\right)\right|\le b$, for each $n$, $\omega$, and $\theta$, i.e. $\left(f_n\right)_{n\in\mathbb{N}}$ is uniformly bounded by $b$.

Motivated by \citet[Example 1.11.5]{van-der-Vaart:2023aa}, we have the fact that for $1<r<q<\infty$,
the sequence $\left(f_{n}^{r}\right)_{n\in\mathbb{N}}$ is AUI with
respect to $\left(\mathrm{Q}_{n}\right)_{n\in\mathbb{N}}$, $\mathrm{P}$-a.s., if
\[
\limsup_{n\to\infty}\int_{\mathbb{T}}\left|f_{n}\left(\omega,\theta\right)\right|^{q}\mathrm{Q}_{n}\left(\omega,\mathrm{d}\theta\right)<\infty\text{,}
\]
holds, for every $\omega$ on a $\mathrm{P}$-a.s. set. This is because, for each $\delta>0$,
\[
\limsup_{n\to\infty}\int_{\mathbb{T}}\left|f_{n}\left(\omega,\theta\right)\right|^{r}
\chi_{\{\tau\in\mathbb{T}: |f_{n}(\omega,\tau)|\ge\delta\}}(\theta)
\mathrm{Q}_{n}\left(\omega,\mathrm{d}\theta\right)\le\delta^{r-q}\limsup_{n\to\infty}\int_{\mathbb{T}}\left|f_{n}\left(\omega,\theta\right)\right|^{q}\mathrm{Q}_{n}\left(\omega,\mathrm{d}\theta\right)\text{,}
\]
where $\lim_{\delta\to\infty}\delta^{r-q}=0$.

\subsection{Convergence of the BPIC and WBIC}

Let $\Pi_{0}\in\mathcal{P}(\mathbb{T})$ denote an arbitrary probability measure. We are now ready to state sufficient conditions for the convergence
of $\mathrm{BPIC}_{n}$ and $\mathrm{WBIC}_{n}$ to their limits.
The following further assumptions are required:
\begin{description}
\item [{A5}] The sequence of average log-likelihoods $\left(n^{-1}\sum_{i=1}^{n}\log\mathrm{p}\left(X_{i}\mid\cdot\right)\right)_{n\in\mathbb{N}}$
is AUI with respect to $\left(\Pi_{n}^{\beta_{n}}\right)_{n\in\mathbb{N}}$,
$\mathrm{P}$-a.s.
\item [{A6}] The sequence $\left(\Pi_{n}^{\beta_{n}}\right)_{n\in\mathbb{N}}$
converges to $\Pi_{0}\in\mathcal{P}\left(\mathbb{T}\right)$, $\mathrm{P}$-a.s.w.
\end{description}
\begin{prop}
Assume that $\left(X_{i}\right)_{i\in\mathbb{N}}$ are IID and that
A1 and A2 hold. If A5 and A6 hold for $\beta_{n}=1$, then 
\begin{align}
\mathrm{BPIC}_{n}\xrightarrow[n\rightarrow\infty]{\mathrm{P}\text{-a.s.}}-2\int_{\mathbb{T}}\mathrm{E}\!\left[\log\mathrm{p}\left(X\mid\theta\right)\right]\Pi_{0}\left(\mathrm{d}\theta\right)\text{.} \label{eq:-bic-convergence}
\end{align}
If A5 and A6 hold with $n\beta_{n}\rightarrow\infty$, then 
\begin{align}
\mathrm{WBIC}_{n}\xrightarrow[n\rightarrow\infty]{\mathrm{P}\text{-a.s.}}-2\int_{\mathbb{T}}\mathrm{E}\!\left[\log\mathrm{p}\left(X\mid\theta\right)\right]\Pi_{0}\left(\mathrm{d}\theta\right)\text{.}  \label{eq:-wbic-convergence}
\end{align}
\end{prop}
If we replace A6 with A3 and A4, we instead obtain the following more specific
outcome.
\begin{cor} \label{cor:-pbic-wbic-delta-convergence}
Assume that $\left(X_{i}\right)_{i\in\mathbb{N}}$ are IID and $\Pi\left(\mathbb{B}_{\rho}\left(\theta_{0}\right)\right)>0$, for every $\rho>0$. If A1--A5
hold for $\beta_{n}=1$, then 
\[
\mathrm{BPIC}_{n}\xrightarrow[n\rightarrow\infty]{\mathrm{P}\text{-a.s.}}-2\mathrm{E}\!\left[\log\mathrm{p}\left(X\mid\theta_{0}\right)\right]\text{.}
\]
If A1--A5 hold with $n\beta_{n}\rightarrow\infty$, then 
\[
\mathrm{WBIC}_{n}\xrightarrow[n\rightarrow\infty]{\mathrm{P}\text{-a.s.}}-2\mathrm{E}\!\left[\log\mathrm{p}\left(X\mid\theta_{0}\right)\right]\text{.}
\]
\end{cor}

\subsection{Convergence of the DIC}

Finally, to obtain the limit of the DIC, we require the convergence
of the second right-hand term of \eqref{eq:-dic}. To this end, we require the following assumption:
\begin{description}
\item [{A7}] The norm $\left\Vert \cdot\right\Vert$ is AUI with respect
to $\left(\Pi_{n}\right)_{n\in\mathbb{N}}$, $\mathrm{P}$-a.s.
\end{description}
Via Proposition \ref{prop:-mean-convergence}, we can show that, when taken together with A1, A2, 
and A6, A7 implies
\begin{align}
\bar{\theta}_{n}\xrightarrow[n\rightarrow\infty]{\mathrm{P}\text{-a.s.}}\bar{\theta}_{0}=\left(\int_{\mathbb{T}}\theta_{1}\Pi_{0}\!\left(\mathrm{d}\theta\right),\dots,\int_{\mathbb{T}}\theta_{j}\Pi_{0}\!\left(\mathrm{d}\theta\right),\dots,\int_{\mathbb{T}}\theta_{p}\Pi_{0}\!\left(\mathrm{d}\theta\right)\right).\label{eq:-mean-convergence}
\end{align}

If A6 is replaced by A3 and A4, then $\bar{\theta}_{0}=\theta_{0}$. By the same argument used prior to \eqref{eq:-delta-measures}, we have
$\delta_{\bar{\theta}_{n}}\xrightarrow[n\rightarrow\infty]{\mathrm{P}\text{-a.s.w.}}\delta_{\bar{\theta}_{0}}$.
Assumptions A1 and A2 then yield the continuous convergence of $\left(n^{-1}\sum_{i=1}^{n}\log\mathrm{p}\left(X_{i}\mid\cdot\right)\right)_{n\in\mathbb{N}}$
to $\mathrm{E}\left[\log\mathrm{p}\left(X\mid\cdot\right)\right]$, $\mathrm{P}$-a.s., and hence
\begin{align}
\frac{2}{n}\sum_{i=1}^{n}\log\mathrm{p}\left(X_{i}\mid\bar{\theta}_{n}\right)\xrightarrow[n\rightarrow\infty]{\mathrm{P}\text{-a.s.}}2\,\mathrm{E}\left[\log\mathrm{p}\left(X\mid\bar{\theta}_{0}\right)\right]\text{.} \label{eq:-mean-convergence-log-like}
\end{align}
We therefore obtain the following convergence result for $\mathrm{DIC}_{n}$.

\begin{prop} \label{prop:-dic-limit}
Assume that $\left(X_{i}\right)_{i\in\mathbb{N}}$ are IID and $\Pi\left(\mathbb{B}_{\rho}\left(\theta_{0}\right)\right)>0$,
for every $\rho>0$. If A1, A2, and A5--A7 hold for $\beta_{n}=1$,
then
\[
\mathrm{DIC}_{n}\xrightarrow[n\rightarrow\infty]{\mathrm{P}\text{-a.s.}}-4\int_{\mathbb{T}}\mathrm{E}\left[\log\mathrm{p}\left(X\mid\theta\right)\right]\Pi_{0}\!\left(\mathrm{d}\theta\right)+2\,\mathrm{E}\left[\log\mathrm{p}\left(X\mid\bar{\theta}_{0}\right)\right]\text{.}
\]
If A6 is replaced by A3 and A4, then
\[
\mathrm{DIC}_{n}\xrightarrow[n\rightarrow\infty]{\mathrm{P}\text{-a.s.}}-2\,\mathrm{E}\left[\log\mathrm{p}\left(X\mid\theta_{0}\right)\right]\text{.}
\]
\end{prop}

\section{Examples}\label{section_examples}

We provide three technical examples, followed by numerical results that illustrate our theoretical findings. These examples are chosen to best demonstrate the workflow of our methodology, rather than to exhaust the scope of its utility. Proofs, derivations, and technical details are collected in Appendix~B and provide useful sketches for practitioners to apply our results to their own settings.

\subsection{The geometric model} \label{sec:the-geometric-model}

We begin with the simple case of the geometric model, whose probability
mass function (PMF) is characterised by
\[
\text{p}\left(x|\theta\right)=\left(1-\theta\right)^{x}\theta\text{,}
\]
for $x\in\mathbb{X}=\mathbb{N}\cup\left\{ 0\right\} $ and $\theta\in\mathbb{T}=\left(0,1\right)$.
We will endow $\theta$ with the uniform prior measure on $\mathbb{T}$;
that is, $\Pi$ has density $\pi\left(\theta\right)=1$ for each $\theta\in\mathbb{T}$.
With the sample mean as $\bar{X}_{n}=n^{-1}\sum_{i=1}^{n}X_{i}$,
let $\bar{\theta}_{n}$ denote the posterior mean value, and suppose
that $\text{E}\left|X\right|<\infty$. Then, for $\beta_{n}>0$ the
information criteria have the forms:
\[
\text{DIC}_{n}=-4\int_{0}^{1}\left\{\bar{X}_{n}\log\left(1-\theta\right)+\log\theta\right\}\Pi_{n}\left(\text{d}\theta\right)
+2\bar{X}_{n}\log\left(1-\bar{\theta}_{n}\right)+2\log\bar{\theta}_{n}\text{,}
\]
\[
\text{BPIC}_{n}=-2\int_{0}^{1}\left\{\bar{X}_{n}\log\left(1-\theta\right)+\log\theta\right\}\Pi_{n}\left(\text{d}\theta\right)+\frac{2}{n}\text{,}
\]
\[
\text{WBIC}_{n}=-2\int_{0}^{1}\left\{\bar{X}_{n}\log\left(1-\theta\right)+\log\theta\right\}\Pi_{n}^{\beta_{n}}\left(\text{d}\theta\right)\text{.}
\]
Let A1--A6 hold with $\theta_{0}=\left\{ 1+\text{E}X\right\}^{-1}$.
Then, Corollary \ref{cor:-pbic-wbic-delta-convergence} and Proposition \ref{prop:-dic-limit} imply that, as $n\to\infty$,
with $n\beta_{n}\to\infty$, we have
\begin{equation}
\text{DIC}_{n},\text{BPIC}_{n},\text{WBIC}_{n}\stackrel[n\to\infty]{\text{P-a.s.}}{\longrightarrow}2\log\left\{ 1+\text{E}X\right\} -2\text{E}X\log\left\{ \frac{\text{E}X}{1+\text{E}X}\right\} \text{.}\label{eq:-000020limit-000020geom}
\end{equation}

\subsection{The normal model} \label{sec:-normal-limit}

We now consider the normal model, defined by the PDF:
\[
\text{p}\left(x|\theta\right)=\left(2\pi\right)^{-p/2}\exp\left\{ -\frac{1}{2}\left\Vert x-\theta\right\Vert ^{2}\right\}\text{,}
\]
for $x\in\mathbb{X}=\mathbb{R}^{p}$ and $\theta\in\mathbb{T}=\mathbb{R}^{p}$,
equipped with the normal prior measure $\Pi$, defined by the PDF
\[
\pi\left(\theta\right)=\left(2\pi\right)^{-p/2}\exp\left\{ -\frac{1}{2}\left\Vert \theta-\mu\right\Vert ^{2}\right\}\text{,}
\]
for some $\mu\in\mathbb{R}^{p}$. We will take $\bar{\theta}_{n}$
to be the multivariate mean of the posterior distribution, and assume
that $X$ has a finite second moment: $\text{E}\left[\left\Vert X\right\Vert ^{2}\right]<\infty$.
Then, for $\beta_{n}>0$, 
\[
\text{DIC}_{n}=p\log\left(2\pi\right)+\frac{2}{n}\int_{\mathbb{R}^{p}}\sum_{i=1}^{n}\left\Vert X_{i}-\theta\right\Vert ^{2}\Pi_{n}\left(\text{d}\theta\right)-\frac{1}{n}\sum_{i=1}^{n}\left\Vert X_{i}-\bar{\theta}_{n}\right\Vert ^{2}\text{,}
\]
\[
\text{BPIC}_{n}=p\log\left(2\pi\right)+\frac{1}{n}\int_{\mathbb{R}^{p}}\sum_{i=1}^{n}\left\Vert X_{i}-\theta\right\Vert ^{2}\Pi_{n}\left(\text{d}\theta\right)+\frac{2p}{n}\text{,}
\]
\begin{equation}
\text{WBIC}_{n}=p\log\left(2\pi\right)+\frac{1}{n}\int_{\mathbb{R}^{p}}\sum_{i=1}^{n}\left\Vert X_{i}-\theta\right\Vert ^{2}\Pi_{n}^{\beta_{n}}\left(\text{d}\theta\right)\text{.}\label{eq:-000020wbic-000020normal}
\end{equation}
Let A1--A6 hold with $\theta_{0}=\text{E}X$. Then, Corollary \ref{cor:-pbic-wbic-delta-convergence} and Proposition \ref{prop:-dic-limit} imply that, as $n\to\infty$, with $n\beta_{n}\to\infty$,
we have
\begin{equation}
\text{DIC}_{n},\text{BPIC}_{n},\text{WBIC}_{n}\stackrel[n\to\infty]{\text{P-a.s.}}{\longrightarrow}p\log\left(2\pi\right)+\text{E}\left[\left\Vert X\right\Vert ^{2}\right]-\left\Vert \text{E}X\right\Vert ^{2}\text{.}\label{eq:-000020normal-000020limit}
\end{equation}

\subsection{The Laplace model} \label{sec:-laplace-model-limits}

We conclude by considering a nonconjugate likelihood--prior pair, taking a Laplace model for $X\in\mathbb{X}=\mathbb{R}$, with PDF
\[
\text{p}\left(x\mid\theta\right)=\frac{1}{2\gamma}\exp\left\{-\frac{\left|x-\mu\right|}{\gamma}\right\}\text{,}
\]
where $\theta=\left(\mu,\gamma\right)\in\mathbb{T}=\left[-m,m\right]\times\left[s^{-1},s\right]$,
with $m>0$ and $s>1$. We will endow $\theta$ with a prior measure
$\Pi$ whose density $\pi$ is strictly positive on $\mathbb{T}$. 

Write $\text{Med}\left(X\right)$ as the median of $X$ and assume
the regularity assumptions that $\text{E}\left|X\right|<\infty$,
$\mu_{0}=\text{Med}\left(X\right)\in\left[-m,m\right]$ and $\gamma_{0}=\text{E}\left|X-\text{Med}\left(X\right)\right|\in\left[s^{-1},s\right]$.
Further write $\bar{\theta}_{n}=\left(\bar{\mu}_{n},\bar{\gamma}_{n}\right)$
to be the posterior mean. Then, for $\beta_{n}>0$, 
\[
\text{DIC}_{n}
=4\int_{\mathbb{T}}\left\{\log\left(2\gamma\right)+\frac{1}{\gamma n}\sum_{i=1}^{n}\left|X_{i}-\mu\right|\right\}\Pi_{n}\left(\text{d}\theta\right)
-2\left\{\log\left(2\bar{\gamma}_{n}\right)+\frac{1}{\bar{\gamma}_{n}n}\sum_{i=1}^{n}\left|X_{i}-\bar{\mu}_{n}\right|\right\}\text{,}
\]
\[
\text{BPIC}_{n}
=2\int_{\mathbb{T}}\left\{\log\left(2\gamma\right)+\frac{1}{\gamma n}\sum_{i=1}^{n}\left|X_{i}-\mu\right|\right\}\Pi_{n}\left(\text{d}\theta\right)+\frac{4}{n}\text{,}
\]
\[
\text{WBIC}_{n}
=2\int_{\mathbb{T}}\left\{\log\left(2\gamma\right)+\frac{1}{\gamma n}\sum_{i=1}^{n}\left|X_{i}-\mu\right|\right\}\Pi_{n}^{\beta_{n}}\left(\text{d}\theta\right)\text{.}
\]
Let A1--A6 hold with $\theta_{0}=\left(\mu_{0},\gamma_{0}\right)$.
Then, Corollary \ref{cor:-pbic-wbic-delta-convergence} and Proposition \ref{prop:-dic-limit} imply that, as $n\to\infty$,
with $n\beta_{n}\to\infty$, we have
\[
\text{DIC}_{n},\text{BPIC}_{n},\text{WBIC}_{n}\stackrel[n\to\infty]{\text{P-a.s.}}{\longrightarrow}2\log\left(2\gamma_{0}\right)+2\text{.}
\]

\subsection{Numerical examples}

The following numerical experiments are implemented in the R programming
language. Code for all results is available on our GitHub page:
\url{https://github.com/hiendn/BayesianIC}.

\subsubsection{DIC for the geometric model}

We continue with the geometric distribution example from Section \ref{sec:the-geometric-model}.
Let us broaden the scenario by replacing the uniform prior with
a beta distribution prior measure $\Pi$, with parameters $\alpha,\beta>0$,
whose law we shall write as $\text{Beta}\left(\alpha,\beta\right)$.
Then, conjugacy implies that the posterior measure $\Pi_{n}$ is the
law $\text{Beta}\left(a_{n},b_{n}\right)$, where $a_{n}=n+\alpha$
and $b_{n}=n\bar{X}_{n}+\beta$. Let $\psi:\mathbb{R}_{>0}\to\mathbb{R}$
denote the digamma function (defined in the Appendix), for which we have the Poincar\'e-type approximation \cite[Eq. 5.11.2]{NIST:2010aa}: for each $a>0$,
\[
\psi\left(a\right)\approx\log a-\frac{1}{2a}\text{.}
\]
Using this fact, we derive in the Appendix the following large-$n$ approximation
of $\text{DIC}_{n}$:
\begin{eqnarray*}
\text{DIC}_{n} & \approx & -2\log\left\{ \frac{n+\alpha}{n+\alpha+n\bar{X}_{n}+\beta}\right\} +2\frac{n\bar{X}_{n}+\beta}{\left(n+\alpha\right)\left(n+\alpha+n\bar{X}_{n}+\beta\right)}\\
 &  & -2\bar{X}_{n}\log\left\{ \frac{n\bar{X}_{n}+\beta}{n+\alpha+n\bar{X}_{n}+\beta}\right\} +2\frac{\left(n+\alpha\right)\bar{X}_{n}}{\left(n\bar{X}_{n}+\beta\right)\left(n+\alpha+n\bar{X}_{n}+\beta\right)}\text{.}
\end{eqnarray*}

Equipped with this computational formula, we design a simulation study
to check the convergence of $\text{DIC}_{n}$ to its large-sample
limit, for a range of values of $\theta_{0}$, with various prior
hyperparameters $\alpha$ and $\beta$, and sample sizes $n$. The
values for $\theta_{0}$ are chosen between $0.1$ and $0.9$, with
prior hyperparameters chosen as pairs with $\alpha,\beta\in\left\{ 1,10,100\right\} $.
Sample sizes are chosen in the range $10^{2}$ to $10^{7}$, and each
scenario is repeated $10$ times. For each $\theta_{0}$, the limiting
value of $\text{DIC}_{n}$ is as per \eqref{eq:-000020limit-000020geom}:
\[
\text{DIC}_{n}\stackrel[n\to\infty]{\text{P-a.s.}}{\longrightarrow}-2\text{E}\left[\log\text{p}\left(X\mid\theta_{0}\right)\right]=-2\left\{ \frac{1-\theta_{0}}{\theta_{0}}\log\left(1-\theta_{0}\right)+\log\left(\theta_{0}\right)\right\}\text{.}
\]
Here, we make the substitution $\theta_{0}=\left\{ 1+\text{E}X\right\}^{-1}$.
From Figure \ref{fig:dic_n_geom}, we see that for smaller $n$, $\text{DIC}_{n}$ can
have high variance, but as $n$ increases, we indeed have convergence
to the theoretical limit for each $\theta_{0}$.

\begin{figure}
    \centering
    \includegraphics[width=0.8\linewidth]{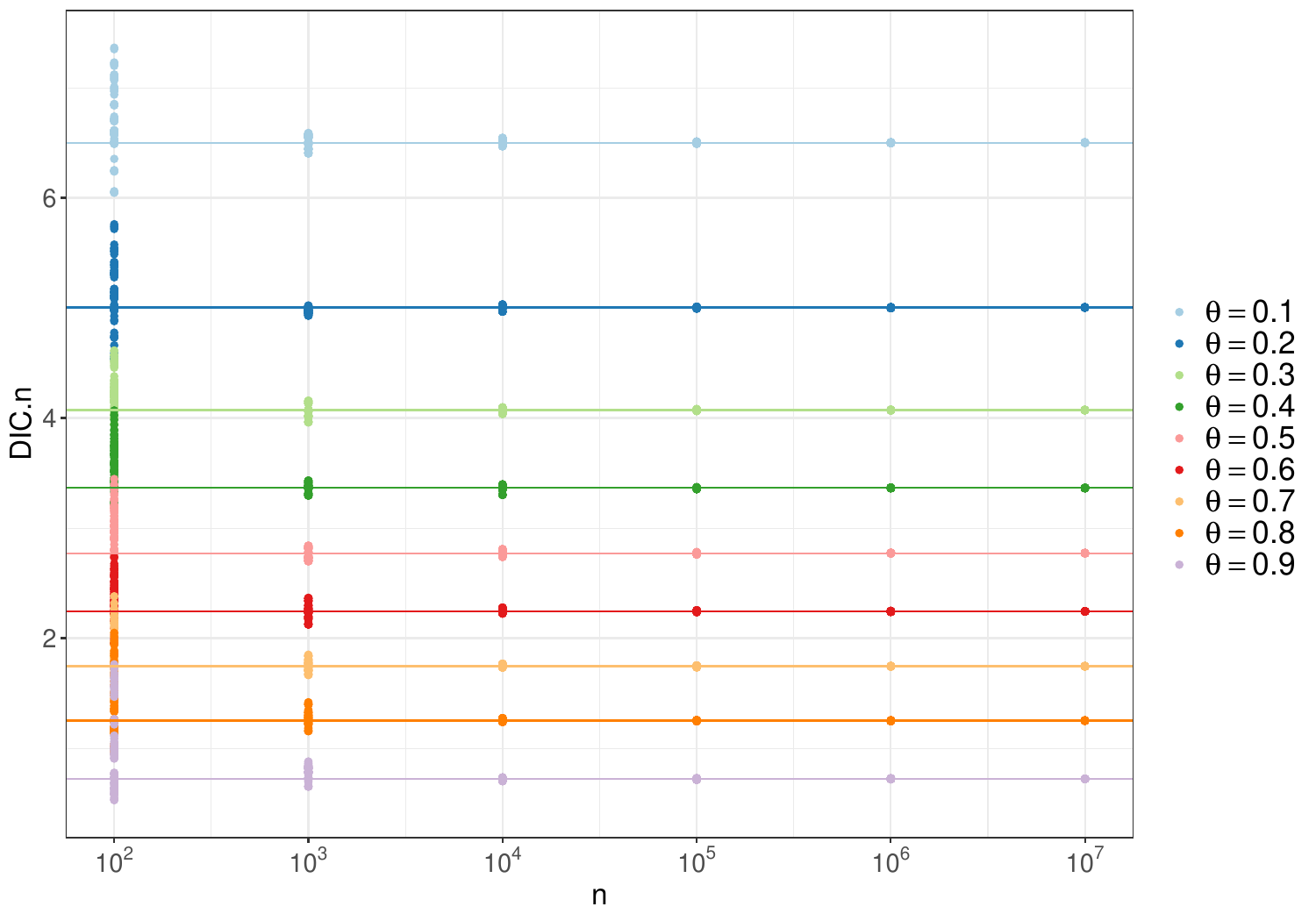}
    \caption{$\text{DIC}_{n}$ values for the geometric model. Points indicate the replicates of $\text{DIC}_{n}$ evaluations for a range of values of $\theta_{0}$ with different sample sizes $n$, while the lines correspond to the limiting value of $\text{DIC}_{n}$ as $n \to \infty$. Both points and lines are coloured according to the true value of $\theta_{0}$ in each case.}
    \label{fig:dic_n_geom}
\end{figure}

\subsubsection{WBIC for the normal model}

We next consider the normal model in Section \ref{sec:-normal-limit} with $p=1$. Here
we investigate the convergence of $\text{WBIC}_{n}$ for different
choices of sequences $\beta_{n}$ that satisfy the condition $n\beta_{n}\to\infty$
as $n\to\infty$. We employ a unit-variance normal prior distribution
$\Pi$ with mean hyperparameter $\mu=0$, for which the power posterior
$\Pi_{n}^{\beta_{n}}$ has the normal law:
\[
\text{N}\left(m_{n},v_{n}\right)\text{,}\qquad m_{n}=\frac{n\beta_{n}\bar{X}_{n}+\mu}{n\beta_{n}+1}\text{,}\quad v_{n}=\frac{1}{n\beta_{n}+1}\text{.}
\]
This is the univariate case of Lemma \ref{lem:pp-normal}.

For any integrable function $g:\mathbb{T}\to\mathbb{R}$, write
\[
\mathrm{E}_{\Pi_{n}^{\beta_{n}}}\bigl[g(\theta)\bigr]=\int g(\theta)\,\Pi_{n}^{\beta_{n}}(\text{d}\theta)\text{.}
\]

In this setting $\text{WBIC}_{n}$ admits a closed form. Writing 
\[
f_{n}(\theta)=\frac{1}{n}\sum_{i=1}^{n}\log\text{p}(X_{i}\mid\theta)=-\frac{1}{2}\log(2\pi)-\frac{1}{2n}\sum_{i=1}^{n}(X_{i}-\theta)^{2}\text{,}
\]
we have 
\[
-2\int f_{n}(\theta)\,\Pi_{n}^{\beta_{n}}(\text{d}\theta)=\log(2\pi)+\frac{1}{n}\sum_{i=1}^{n}X_{i}^{2}-2\bar{X}_{n}\,\mathrm{E}_{\Pi_{n}^{\beta_{n}}}[\theta]+\mathrm{E}_{\Pi_{n}^{\beta_{n}}}[\theta^{2}]\text{,}
\]
and since $\Pi_{n}^{\beta_{n}}$ is the law $\mathrm{N}(m_{n},v_{n})$,
$\mathrm{E}_{\Pi_{n}^{\beta_{n}}}[\theta]=m_{n}$ and $\mathrm{E}_{\Pi_{n}^{\beta_{n}}}[\theta^{2}]=m_{n}^{2}+v_{n}$.
Therefore 
\begin{equation}
\text{WBIC}_{n}=\log(2\pi)+\frac{1}{n}\sum_{i=1}^{n}X_{i}^{2}-2\bar{X}_{n}m_{n}+m_{n}^{2}+v_{n}\text{,}\label{eq:wbic-normal-closed-form}
\end{equation}
which holds for all $n$ and all $\beta_{n}>0$.

For completeness, the same calculation extends to $p\ge1$ with an
$\text{N}_{p}(\mu,\mathbf{I})$ prior; i.e., with 
\[
f_{n}(\theta)=-\frac{p}{2}\log(2\pi)-\frac{1}{2n}\sum_{i=1}^{n}\Vert X_{i}-\theta\Vert^{2}\text{,}
\]
the power posterior is $\mathrm{N}(m_{n},\Sigma_{n})$ with $m_{n}=(n\beta_{n}\bar{X}_{n}+\mu)/(n\beta_{n}+1)$
and $\Sigma_{n}=(n\beta_{n}+1)^{-1}\mathbf{I}$, and 
\[
\text{WBIC}_{n}=p\log(2\pi)+\mathrm{tr}\left(\frac{1}{n}\sum_{i=1}^{n}X_{i}X_{i}^{\top}\right)-2\bar{X}_{n}^{\top}m_{n}+\Vert m_{n}\Vert^{2}+\mathrm{tr}(\Sigma_{n})\text{.}
\]

Using (\ref{eq:wbic-normal-closed-form}), we report $\text{WBIC}_{n}$
under the original power sequence $\beta_{n}=1/\log n$ of \citet{watanabe2013widely},
along with alternative choices $\beta_{n}=1/\log\log n$, $\beta_{n}=1$,
and $\beta_{n}=1/\sqrt{n}$. We also consider $\beta_{n}=1/n$ and
$\beta_{n}=1/(n\log n)$ to probe the necessity of the condition
$n\beta_{n}\to\infty$. Figure \ref{fig:wbic-000020sim} displays
our results for $10$ replicates for each sample size $n$ ranging
from $10^{1}$ to $10^{5}$, with data simulated from $\mathrm{N}(\theta_{0},1)$
and $\theta_{0}=1$.

\begin{figure}
\begin{centering}
\includegraphics[width=15cm]{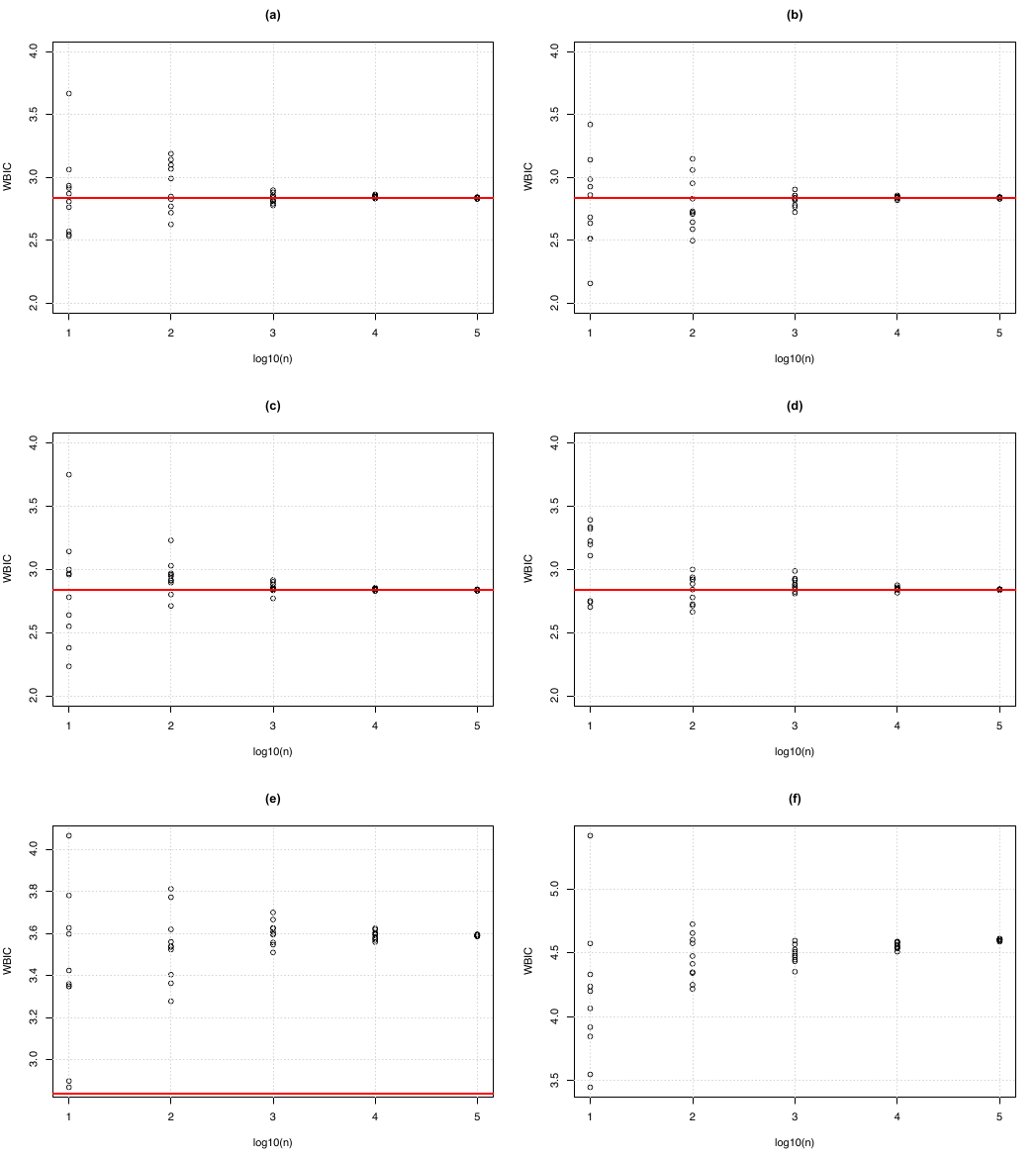}
\par\end{centering}
\caption{\protect\label{fig:wbic-000020sim}Simulated $\text{WBIC}_{n}$ values
under various choices of $\beta_{n}$: (a) $1/\log n$, (b) $1/\log\log n$,
(c) $1$, (d) $1/\sqrt{n}$, (e) $1/n$, and (f) $1/(n\log n)$.
Points indicate the replicates of $\text{WBIC}_{n}$ evaluations, and the solid lines indicate the theoretical limit under Corollary \ref{cor:-pbic-wbic-delta-convergence}.}
\end{figure}

The simulation results are consistent with the theory: whenever $n\beta_{n}\to\infty$,
$\text{WBIC}_{n}$ appears to converge to the predicted limit (\ref{eq:-000020normal-000020limit}),
\[
-2\mathrm{E}\left[\log\text{p}\left(X\mid\theta_{0}\right)\right]=\log(2\pi)+1\text{.}
\]
When $n\beta_{n}$ fails to diverge (Figure \ref{fig:wbic-000020sim}
(e, f)), the trajectories either do not approach this limit or stabilise
markedly more slowly, providing evidence for the necessity of the
growth condition $n\beta_{n}\to\infty$. We emphasise that our results
furnish sufficient conditions for convergence and do not make claims
outside these assumptions.

\section{Discussion}\label{section_discussion}

In this work, we thoroughly investigate the limiting behaviour of several
Bayesian model selection criteria--namely, the DIC, BPIC, and WBIC--as
sample sizes approach infinity, under the condition that posterior
distributions are consistent. We employ generalisations of the dominated
convergence theorem, which have not been widely utilised in the literature,
and we establish several general results regarding Bayesian posterior
consistency, whose utility extends beyond the scope of this study.

It is important to note that our analysis is qualitative in nature
and does not provide finite sample predictions for the behaviour of
the aforementioned information criteria. Specifically, our investigation
does not offer comparative recommendations regarding the utility or
merit of these criteria for model or variable selection.

In the frequentist setting, \citet{Sin:1996aa}, as well as subsequent works
by \citet{nguyen2024panic} and \citet{westerhout2024asymptotic}, have explored general conditions under which
information criteria based on minima of risk functions can achieve
model selection consistency. This consistency is understood as the
asymptotic selection of the minimal complexity model that minimises
risk as $n\to\infty$. Conversely, finite sample properties of information-criteria-like,
penalty-based methods have been studied in works such as those by
\citet{Massart2007} and \citet{koltchinskii2011oracle}, which provide broadly applicable bounds on
the risk incurred by these methods. Within this context, our results
provide only a minimal qualitative guarantee. Specifically, in the terminology of \citet{Claeskens2008}, our findings establish asymptotic efficiency (i.e., AIC-like behaviour), rather than model-selection consistency: when used to compare models, these criteria will almost surely select, in large samples, a model that minimises the expected negative log-likelihood among the candidates. This efficiency guarantee matches that of the AIC, and is akin to selecting models by their sample log-likelihood values.
Further research
is necessary to derive conditions under which Bayesian model selection
criteria, based on posterior integrated likelihoods rather than maximum
likelihoods, can achieve consistent parsimonious model selection in the sense of \cite{Sin:1996aa} and \cite{Claeskens2008}, although we note that parsimony in Bayesian model selection can be achieved under other frameworks, such as that of \cite{ChipmanGeorgeMcCulloch2001PracticalBayes}, \cite{friston2011post}, \cite{JohnsonRossell2012JASA}, \cite{BayarriBergerForteGarciaDonato2012AOS}, and \cite{devashish2025bayesian}.

Additionally, while our primary analysis focuses on the posterior
distribution in the role of the random measure $\Pi_{n}$ and the
log-likelihood function as the integrand, our methodology extends broadly
to the construction of information criteria based on more general
objects. For instance, our framework can be applied to analyse the
behaviour of the variational inference information criteria proposed
by \citet{you2014variational}, generalised posterior measures with corresponding risk
functions, or scenarios in which the posterior is replaced
by Bernstein--von Mises approximations or approximate Bayesian posterior
measures that exhibit consistency, such as in \citet{nguyen2025revisiting}. Thus, our work provides
a flexible template for analysing a broad class of Bayesian model selection approaches. In \cite{mclatchie2025predictive}, the authors show that, under posterior consistency,
the temperature parameter $\beta$ in the power posterior predictive
measure does not affect the large-sample limit: uniformly over $\beta$
on positive compact sets (and even for sequences $\beta_{n}\to0$
with $n\beta_{n}\to\infty$), the power posterior predictive measure
coincides asymptotically with the plug-in predictive measure. Our
results concern different objects, namely the posterior and posterior-based
information criteria (\text{WBIC}, \text{BPIC}, \text{DIC}),
and establish exact large-sample limits under sufficient conditions
that include $n\beta_{n}\to\infty$. Thus, rather than mirroring \cite{mclatchie2025predictive}, our theory complements theirs: both indicate that temperature
tuning has vanishing asymptotic impact, but on different functionals
and via different techniques. Exploring when these two perspectives
align is a natural direction for future work. 

Finally, given our focus on information criteria for parametric models,
we have focused on establishing sufficient conditions for posterior
consistency within this context. However, with modifications, these
conditions could be generalised to align with the deterministic metric
space framework of \citet{miller2021asymptotic}. Furthermore, our analysis raises interesting
questions about alternative definitions of posterior consistency, such
as allowing posterior mass to converge on a set rather than a single
point. Such a perspective could allow for the development of limit
theorems for information criteria in non-identifiable models. We aim
to extend our proofs to these settings in future work.

\begin{center}
\textbf{{\Large Appendix}}
\end{center}

\appendix

\section{Proofs and technical results}\label{section_proofs}

\subsection{Proof of Proposition \ref{prop:-consistency-of-beta-posterior}}\label{proof_prop_consistency_beta_posterior}
As promised, we shall prove a more general result that implies Proposition
\ref{prop:-consistency-of-beta-posterior}. Following Section \ref{sec:introduction},
we define a sequence of data-dependent utility functions $\left(u_{n}\right)_{n\in\mathbb{N}}$,
where, for each $\omega\in\Omega$ and $n\in\mathbb{N}$:
\[
u_{n}\left(\omega,\theta\right)=u_{n}\left(X_{1}(\omega),\dots,X_{n}(\omega);\theta\right)\text{.}
\]
We make the following assumptions regarding $\left(u_{n}\right)_{n\in\mathbb{N}}$,
mirroring assumptions A1--A4. 
\begin{description}
\item[{B1}] For each $n\in\mathbb{N}$, the utility $u_{n}$ is Carath\'{e}odory
in the sense that $u_n\left(\cdot,\theta\right)$ is measurable for
each $\theta\in\mathbb{T}$ and $u_n\left(\omega,\cdot\right)$ is continuous
for each $\omega\in\Omega$. 
\item[{B2}] There exists a deterministic function $u:\mathbb{T}\rightarrow\mathbb{R}$
and a P-a.s. set on which, for every compact $\mathbb{K}\subset\mathbb{T}$,
\begin{equation}
\sup_{\theta\in\mathbb{K}}\left|u_{n}\left(\omega,\theta\right)-u\left(\theta\right)\right|\xrightarrow[n\rightarrow\infty]{}0\text{.}\label{eq:-uniform-convergence-of-utilities}
\end{equation}
\item[{B3}] There exists a unique maximiser $\theta_{0}\in\mathbb{T}$
of $u\left(\theta\right)$, in the sense that, for every $\theta\in\mathbb{T}\backslash\{\theta_0\}$,
\[u(\theta)<u(\theta_0)\text{.}\]
\item[{B4}] There exists a compact set $\mathbb{S}\subsetneq\mathbb{T}$ containing $\theta_0$
such that, on a P-a.s. set,
\[
\limsup_{n\rightarrow\infty}\sup_{\theta\in\mathbb{S}^{\text{c}}}u_{n}\left(\omega,\theta\right)<u\left(\theta_{0}\right)\text{.}
\]
\end{description}
Let $\left(\varPi_{n}\right)_{n\in\mathbb{N}}$ be a sequence of random
measures, where $\varPi_{n}:\Omega\times\mathfrak{B}_{\mathbb{T}}\rightarrow\mathbb{R}_{\ge0}$
for each $n\in\mathbb{N}$, defined by
\[
\varPi_{n}\left(\omega,\mathbb{A}\right)=\frac{\int_{\mathbb{A}}\exp\left(\gamma_{n}u_{n}\left(\omega,\theta\right)\right)\Pi\left(\text{d}\theta\right)}{\int_{\mathbb{T}}\exp\left(\gamma_{n}u_{n}\left(\omega,\theta\right)\right)\Pi\left(\text{d}\theta\right)}\text{,}
\]
for each $\omega\in\Omega$, where $\mathbb{A}\in\mathfrak{B}_{\mathbb{T}}$. Here $\left(\gamma_{n}\right)_{n\in\mathbb{N}}\subset\mathbb{R}_{>0}$
diverges to infinity. The measures $\varPi_{n}$ extend the generalised or Gibbs posterior measures and appear in
the works of \citet{Zhang:2006ab}, \citet{Bissiri:2016aa}, \cite{miller2021asymptotic}, and \citet{Knoblauch:2022aa},
among others. We prove the following result.

\begin{thm}
\label{thm:-generalised-posterior-consistency}
If $\Pi\left(\mathbb{B}_{\rho}\left(\theta_{0}\right)\right)>0$ for every $\rho>0$, and B1--B4 hold, then the sequence $\left(\varPi_{n}\right)_{n\in\mathbb{N}}$
is consistent in the sense that, for every $\epsilon>0$ and for every $\omega$ on a $\mathrm{P}$-a.s. set,
\[
\varPi_{n}\left(\omega,\bar{\mathbb{B}}_{\epsilon}\left(\theta_{0}\right)\right)\xrightarrow[n\rightarrow\infty]{}1\text{.}
\]
\end{thm}
\begin{proof}
We start by selecting an arbitrary $\epsilon>0$. By B1 and B2, $u$
is continuous on compact sets, and by B3, $\theta_{0}$ is the unique
maximiser of $u$ on $\mathbb{T}$.

Consider the compact set $\mathbb{S}\cap\mathbb{B}_{\epsilon}^{\text{c}}(\theta_{0})$.
By uniqueness in B3, 
\[
\Delta=u(\theta_{0})-\sup_{\theta\in\mathbb{S}\cap\mathbb{B}_{\epsilon}^{\text{c}}(\theta_{0})}u(\theta)>0\text{.}
\]
Pick any $\delta\in(0,\Delta/4)$. Let $\varOmega_{\mathrm{unif}}$ denote the $\mathrm{P}$-a.s. event
from B2 on which 
\[
\sup_{\theta\in\mathbb{K}}\bigl|u_{n}(\omega,\theta)-u(\theta)\bigr|\xrightarrow[n\to\infty]{}0\quad\text{for every compact }\mathbb{K}\subset\mathbb{T},
\]
and let $\varOmega_{\mathrm{tail}}$ denote the $\mathrm{P}$-a.s.
event from B4 on which there exists a random margin $\kappa(\omega)>0$
such that 
\[
\limsup_{n\to\infty}\sup_{\theta\in\mathbb{S}^{\text{c}}}u_{n}(\omega,\theta)\le u(\theta_{0})-\kappa(\omega)\text{.}
\]
We shall work on $\varOmega=\varOmega_{\mathrm{unif}}\cap\varOmega_{\mathrm{tail}}$ and for each $\omega\in\varOmega$, fix  $\delta_{\omega}\in\bigl(0,\min\{\delta,\kappa(\omega)/3\}\bigr)$.
As $\mathbb T$ is the intersection of an open and closed set, it is locally compact. Hence for $\rho$ sufficiently small, $\overline{\mathbb{B}}_{\rho}(\theta_0)$ is compact. Then,
by continuity of $u$ at $\theta_{0}$,
there exists $\rho(\omega)\in(0,\epsilon)$ such that 
\[
\inf_{\theta\in\overline{\mathbb{B}}_{\rho(\omega)}(\theta_{0})}u(\theta)\ge u(\theta_{0})-\delta_\omega\text{.}
\]
Combining these two expressions yields 
\[
\sup_{\theta\in\mathbb{S}\cap\mathbb{B}_{\epsilon}^{\text{c}}(\theta_{0})}u(\theta)\le u(\theta_{0})-4\delta_\omega\qquad\text{and}\qquad\inf_{\theta\in\overline{\mathbb{B}}_{\rho(\omega)}(\theta_{0})}u(\theta)\ge u(\theta_{0})-\delta_\omega\text{.}
\]

Then there exists $N(\omega)$ such that, for all $n\ge N(\omega)$,
\begin{align}
\inf_{\theta\in\overline{\mathbb{B}}_{\rho(\omega)}(\theta_{0})}u_{n}(\omega,\theta) & \ge\inf_{\theta\in\overline{\mathbb{B}}_{\rho(\omega)}(\theta_{0})}u(\theta)-\delta_{\omega}\ge u(\theta_{0})-2\delta_{\omega}\text{,}\label{eq:ball-lower-alt}\\
\sup_{\theta\in\mathbb{S}\cap\mathbb{B}_{\epsilon}^{\text{c}}(\theta_{0})}u_{n}(\omega,\theta) & \le\sup_{\theta\in\mathbb{S}\cap\mathbb{B}_{\epsilon}^{\text{c}}(\theta_{0})}u(\theta)+\delta_{\omega}\ \le\ u(\theta_{0})-3\delta_{\omega}\text{,}\label{eq:ring-upper-alt}\\
\sup_{\theta\in\mathbb{S}^{\text{c}}}u_{n}(\omega,\theta) & \le u(\theta_{0})-3\delta_{\omega}\text{.}\label{eq:tail-upper-alt}
\end{align}

For $\mathbb{A}\in\mathfrak{B}_{\mathbb{T}}$, write 
\[
I_{n}(\omega;\mathbb{A})=\int_{\mathbb{A}}\exp\bigl(\gamma_{n}u_{n}(\omega,\theta)\bigr)\,\Pi(\text{d}\theta)\text{.}
\]
By prior positivity, $\Pi\bigl(\overline{\mathbb{B}}_{\rho(\omega)}(\theta_{0})\bigr)>0$, and using (\ref{eq:ball-lower-alt}), we have 
\begin{align}
I_{n}\bigl(\omega;\overline{\mathbb{B}}_{\rho(\omega)}(\theta_{0})\bigr) & =\int_{\overline{\mathbb{B}}_{\rho(\omega)}(\theta_{0})}\exp\bigl(\gamma_{n}u_{n}(\omega,\theta)\bigr)\,\Pi(\text{d}\theta)\nonumber \\
 & \ge\int_{\overline{\mathbb{B}}_{\rho(\omega)}(\theta_{0})}\exp\Bigl(\gamma_{n}\inf_{\vartheta\in\overline{\mathbb{B}}_{\rho(\omega)}(\theta_{0})}u_{n}(\omega,\vartheta)\Bigr)\,\Pi(\text{d}\theta)\nonumber \\
 & =\exp\Bigl(\gamma_{n}\inf_{\vartheta\in\overline{\mathbb{B}}_{\rho(\omega)}(\theta_{0})}u_{n}(\omega,\vartheta)\Bigr)\ \Pi\bigl(\overline{\mathbb{B}}_{\rho(\omega)}(\theta_{0})\bigr)\nonumber \\
 & \ge\Pi\bigl(\overline{\mathbb{B}}_{\rho(\omega)}(\theta_{0})\bigr)\ \exp\bigl(\gamma_{n}(u(\theta_{0})-2\delta_{\omega})\bigr)\text{.}\label{eq:denom-lb-alt}
\end{align}
Then, using the inclusion $\overline{\mathbb{B}}_{\epsilon}^{\text{c}}(\theta_{0})\subseteq\bigl(\mathbb{S}\cap\mathbb{B}_{\epsilon}^{\text{c}}(\theta_{0})\bigr)\cup\mathbb{S}^{\text{c}}$,
we have:
\begin{eqnarray}
I_{n}\bigl(\omega;\overline{\mathbb{B}}_{\epsilon}^{\text{c}}(\theta_{0})\bigr) & \le & \int_{\mathbb{S}\cap\mathbb{B}_{\epsilon}^{\text{c}}(\theta_{0})}\exp\bigl(\gamma_{n}u_{n}(\omega,\theta)\bigr)\Pi(\text{d}\theta)+\int_{\mathbb{S}^{\text{c}}}\exp\bigl(\gamma_{n}u_{n}(\omega,\theta)\bigr)\Pi(\text{d}\theta)\nonumber \\
 & \le & \int_{\mathbb{S}\cap\mathbb{B}_{\epsilon}^{\text{c}}(\theta_{0})}\exp\Bigl(\gamma_{n}\sup_{\vartheta\in\mathbb{S}\cap\mathbb{B}_{\epsilon}^{\text{c}}(\theta_{0})}u_{n}(\omega,\vartheta)\Bigr)\Pi(\text{d}\theta)\nonumber \\
 &  & +\int_{\mathbb{S}^{\text{c}}}\exp\Bigl(\gamma_{n}\sup_{\vartheta\in\mathbb{S}^{\text{c}}}u_{n}(\omega,\vartheta)\Bigr)\Pi(\text{d}\theta)\nonumber \\
 & = & \Pi\bigl(\mathbb{S}\cap\mathbb{B}_{\epsilon}^{\text{c}}(\theta_{0})\bigr)\exp\Bigl(\gamma_{n}\sup_{\vartheta\in\mathbb{S}\cap\mathbb{B}_{\epsilon}^{\text{c}}(\theta_{0})}u_{n}(\omega,\vartheta)\Bigr)\nonumber \\
 &  & +\Pi\bigl(\mathbb{S}^{\text{c}}\bigr)\exp\Bigl(\gamma_{n}\sup_{\vartheta\in\mathbb{S}^{\text{c}}}u_{n}(\omega,\vartheta)\Bigr)\nonumber \\
 & \le & \Pi\bigl(\mathbb{S}\cap\mathbb{B}_{\epsilon}^{\text{c}}(\theta_{0})\bigr)\exp\bigl(\gamma_{n}(u(\theta_{0})-3\delta_{\omega})\bigr)\nonumber \\
 &  & +\Pi\bigl(\mathbb{S}^{\text{c}}\bigr)\exp\bigl(\gamma_{n}(u(\theta_{0})-3\delta_{\omega})\bigr)\nonumber \\
 & = & \left\{\Pi\bigl(\mathbb{S}\cap\mathbb{B}_{\epsilon}^{\text{c}}(\theta_{0})\bigr)+\Pi(\mathbb{S}^{\text{c}})\right\}\exp\bigl(\gamma_{n}(u(\theta_{0})-3\delta_{\omega})\bigr)\nonumber \\
 & \le & \Pi(\mathbb{T})\exp\bigl(\gamma_{n}(u(\theta_{0})-3\delta_{\omega})\bigr)\text{,}\label{eq:numer-ub-alt}
\end{eqnarray}
where the final inequality uses (\ref{eq:ring-upper-alt}) and (\ref{eq:tail-upper-alt}).
Combining (\ref{eq:denom-lb-alt}) and (\ref{eq:numer-ub-alt}), for
all $n\ge N(\omega)$, 
\[
\frac{I_{n}\bigl(\omega;\overline{\mathbb{B}}_{\epsilon}^{\text{c}}(\theta_{0})\bigr)}{I_{n}\bigl(\omega;\overline{\mathbb{B}}_{\rho(\omega)}(\theta_{0})\bigr)}\ \le\ \frac{\Pi(\mathbb{T})}{\Pi\bigl(\overline{\mathbb{B}}_{\rho(\omega)}(\theta_{0})\bigr)}\ \exp\bigl(-\gamma_{n}\delta_{\omega}\bigr)\xrightarrow[n\to\infty]{}0\text{,}
\]
and hence 
\[
\varPi_{n}\bigl(\omega,\overline{\mathbb{B}}_{\epsilon}(\theta_{0})\bigr)\ge\left\{ 1+\frac{I_{n}\bigl(\omega;\overline{\mathbb{B}}_{\epsilon}^{\text{c}}(\theta_{0})\bigr)}{I_{n}\bigl(\omega;\overline{\mathbb{B}}_{\rho(\omega)}(\theta_{0})\bigr)}\right\} ^{-1}\xrightarrow[n\to\infty]{}1\text{.}
\]
Since the above holds for every $\omega$ on the $\mathrm{P}$-a.s.
event $\varOmega$, the claim follows. 
\end{proof}

We note that this proof follows essentially the same steps as the proof of \citet[Thm. 1.3.4]{Ghosh:2003aa}, which is specialised to the situation where
\begin{equation}
u_{n}\left(\omega,\theta\right)=\frac{1}{n}\sum_{i=1}^{n}\log\text{p}\left(X_{i}(\omega)\mid\theta\right)\label{eq:-utility-average-like}
\end{equation}
is the log-likelihood of an IID sequence $\left(X_{i}\right)_{i\in\left[n\right]}$
and $\gamma_{n}=n$. We also point out that our result is complementary
to the abstract methods of \citet[Sec. 2]{miller2021asymptotic} (in particular, Theorem 3(1,2)),
where the author provides sufficient conditions for the posterior consistency of non-stochastic generalised posterior
measures in the case when $\gamma_{n}=n$. In fact, our assumptions
B1--B4 can be read as stochastic counterparts of \citet[Thm. 3(2)]{miller2021asymptotic}, specialised to $\mathbb{T}\subset\mathbb{R}^{p}$.
Our proof establishes the required stochastic assumptions in the case
when $\mathbb{T}\subset\mathbb{R}^{p}$ and makes the generalisation
to generic sequences $\left(\gamma_{n}\right)_{n\in\mathbb{N}}$. To prove Proposition \ref{prop:-consistency-of-beta-posterior}
using Theorem \ref{thm:-generalised-posterior-consistency}, we observe that B1--B4 are satisfied by A1--A4 when taking $u_{n}$ of the form \eqref{eq:-utility-average-like}.
The result is then implied by setting $\gamma_{n}=n\beta_{n}$. Next, we prove a minor extension to Theorem \ref{thm:-generalised-posterior-consistency}.

\begin{prop}
\label{prop:eta-rescale-short}
Assume the hypotheses of Theorem \ref{thm:-generalised-posterior-consistency}.
Let $\eta:\mathbb{R}\to\mathbb{R}$ be monotone increasing and
satisfy
\[
\eta(a\gamma_{n})-\eta(b\gamma_{n})\xrightarrow[n\to\infty]{}\infty,\qquad\text{for every }a>b,
\]
for the sequence $(\gamma_{n})_{n\in\mathbb{N}}$ defining $\varPi_n$. Define $\varPi_{n}^{(\eta)}$
by replacing $\exp\bigl(\gamma_{n}u_{n}(\omega,\theta)\bigr)$ with
$\exp\bigl(\eta(\gamma_{n}u_{n}(\omega,\theta))\bigr)$ in the numerator
and denominator of $\varPi_{n}$. Then the conclusion of Theorem \ref{thm:-generalised-posterior-consistency}
holds for $\varPi_{n}^{(\eta)}$, i.e., for every $\epsilon>0$,
\[
\varPi_{n}^{(\eta)}\bigl(\omega,\overline{\mathbb{B}}_{\epsilon}(\theta_{0})\bigr)\xrightarrow[n\to\infty]{\mathrm{P}\text{-a.s.}}1\text{.}
\]
\end{prop}

\begin{proof}
Repeat the proof of Theorem \ref{thm:-generalised-posterior-consistency}
verbatim, recalling that there exists a $\mathrm{P}$-a.s. event $\varOmega$
on which, for each $\omega\in\varOmega$, there exists some $\delta_{\omega}>0$
such that, for all large $n$ and for some $\rho(\omega)\in(0,\epsilon)$,
\[
\inf_{\theta\in\overline{\mathbb{B}}_{\rho(\omega)}(\theta_{0})}u_{n}(\omega,\theta)\ge u(\theta_{0})-2\delta_{\omega},\qquad
\sup_{\theta\in(\mathbb{S}\cap\mathbb{B}_{\epsilon}^{\text{c}}(\theta_{0}))\cup\mathbb{S}^{\text{c}}}u_{n}(\omega,\theta)\le u(\theta_{0})-3\delta_{\omega}.
\]
With $\varPi_{n}^{(\eta)}$, the same ratio bound is obtained but with
\[
\exp\bigl(\gamma_{n}(u(\theta_{0})-3\delta_{\omega})\bigr)\big/\exp\bigl(\gamma_{n}(u(\theta_{0})-2\delta_{\omega})\bigr)
\]
replaced by
\[
\exp\Bigl(\eta\bigl(\gamma_{n}(u(\theta_{0})-3\delta_{\omega})\bigr)-\eta\bigl(\gamma_{n}(u(\theta_{0})-2\delta_{\omega})\bigr)\Bigr)\text{.}
\]
Since $u(\theta_{0})-2\delta_{\omega}>u(\theta_{0})-3\delta_{\omega}$,
the assumption on $\eta$ gives that the latter term tends to $0$,
and the conclusion follows exactly as before.
\end{proof}

\begin{rem}
A simple example of an $\eta$ that satisfies the above result is any increasing homogeneous
function of odd degree, such as $\eta\left(x\right)=x^{2k+1}$, for $k\in\mathbb{N}\cup\left\{ 0\right\}$.
We note that, to the best of our knowledge, this fact has not been
reported in any prior works, although we also do not know of any immediate
application.
\end{rem}

\begin{rem}
Note that Theorem \ref{thm:-generalised-posterior-consistency} establishes Bayesian consistency of the sequence
of random measures $\left(\varPi_{n}\right)_{n\in\mathbb{N}}$. Namely,
it establishes the weak convergence, $\mathrm{P}$-a.s., of $\left(\varPi_{n}\right)_{n\in\mathbb{N}}$
to the measure $\delta_{\theta_{0}}$. Such results are best compared
with parametric consistency results in $M$-estimator and extremal estimator
theory (see, e.g., \citealp[Thm. 5.5]{Shapiro2021}). Alternatively, one may consider
the convergence of the posterior distribution of the local (scaled
and centred) parameter $\vartheta=\sqrt{n}\left(\theta-\theta_{0}\right)$,
which typically converges in total variation to a normal distribution,
under appropriate regularity conditions. These results are commonly
referred to as Bernstein--von Mises theorems and are distinct from
posterior consistency theorems; such results can be analogously
compared to asymptotic normality theorems in $M$-estimator and extremal
estimator theory (see, e.g., \citealp[Thm. 5.8]{Shapiro2021}). Various Bernstein--von Mises theorems can be found, for example, in \citet[Thm. 7.89]{schervish2012theory}, \citet[Thm. 21]{ferguson2017course}, \citet[Thm. 1.4.2]{Ghosh:2003aa}, \citet[Thm. 12.1]{ghosal2017fundamentals}, and \citet[Sec. 3]{miller2021asymptotic}, among many
other works. Like asymptotic normality results, and compared with the corresponding Bayesian consistency results, Bernstein--von Mises
theorems typically require much stronger assumptions regarding the smoothness
and growth of the utility function around $\theta_{0}$
to obtain the strong normal limit theorems. Thus, analogously
again, the minimal conditions under which Bayesian consistency or Bernstein--von
Mises theorems hold are typically of independent interest.
Lastly, we note that Bernstein--von Mises theorems make
conclusions regarding the posterior distributions of the local parameter
$\vartheta$, rather than those of $\theta$, and are thus not required
in our setting.
\end{rem}

When $\mathbb{T}$ is convex, we say that $u:\mathbb{T}\to\mathbb{R}$
is quasiconvex if for each $\theta,\tau\in\mathbb{T}$ and $\lambda\in\left[0,1\right]$,
\[
u\left(\lambda\theta+\left(1-\lambda\right)\tau\right)\le\max\left\{ u\left(\theta\right),u\left(\tau\right)\right\}\text{.}
\]
Furthermore, we say that $u$ is quasiconcave if $-u$ is quasiconvex.
The class of quasiconvex functions includes the convex functions, and
forms a generalisation that preserves the property of convex sublevel
sets. Using the notion of quasiconvexity, we provide the following
result for verifying B4, providing a generalisation of the convexity-based
result of \citet[Lem. 27]{miller2021asymptotic}, and when taken together with Theorem \ref{thm:-generalised-posterior-consistency}, we obtain
a stochastic quasiconvexity-based version of \citet[Thm. 3(3)]{miller2021asymptotic}.
\begin{lem} \label{lem:-quasiconcave}
Let $\mathbb{T}\subset\mathbb{R}^{p}$ and assume either that $\theta_{0}$ is in the interior of $\mathbb{T}$, or that $\mathbb{T}$ is the intersection of an open and closed set.
Suppose that $\left(u_{n}\right)_{n\in\mathbb{N}}$ is a sequence
of functions where, for each $n\in\mathbb{N}$, $u_{n}:\mathbb{T}\to\mathbb{R}$ is continuous
and quasiconcave, and
\begin{equation}
\sup_{\theta\in\mathbb{K}}\left|u_{n}\left(\theta\right)-u\left(\theta\right)\right|\underset{n\rightarrow\infty}{\longrightarrow}0\text{,}\label{eq:-uniform-convergence-quasiconcave}
\end{equation}
for each compact $\mathbb{K}\subset\mathbb{T}$, for some $u:\mathbb{T}\rightarrow\mathbb{R}$.
Then:
\begin{description}
\item[{(i)}] The function $u$ is continuous and quasiconcave.
\item[{(ii)}] If $\theta_0$ satisfies $u\left(\theta\right)<u\left(\theta_{0}\right)$ for
every $\theta\in\mathbb{T}\backslash\left\{\theta_{0}\right\}$,
then, for every $\epsilon>0$,
\[
\limsup_{n\rightarrow\infty}\sup_{\theta\in\mathbb{T}\backslash\mathbb{B}_{\epsilon}\left(\theta_{0}\right)}u_{n}\left(\theta\right)<u\left(\theta_{0}\right)\text{.}
\]
\end{description}
\end{lem}

\begin{proof}
(i) We first show that $u:\mathbb{T}\to\mathbb{R}$ is quasiconcave.
Fix $\theta,\tau\in\mathbb{T}$ and $\lambda\in[0,1]$, and let $\xi=\lambda\theta+(1-\lambda)\tau$.
The set $\{\theta,\tau,\xi\}$ is compact, so (\ref{eq:-uniform-convergence-quasiconcave})
gives $u_{n}\to u$ uniformly on it, hence pointwise at each of the
three points. Since $u_{n}$ is quasiconcave, $-u_{n}$ is quasiconvex
and 
\[
-u_{n}(\xi)\le\max\{-u_{n}(\theta),-u_{n}(\tau)\}\text{.}
\]
Letting $n\to\infty$ and using the continuity of $(a,b)\mapsto\max\{a,b\}$
yields 
\[
-u(\xi)\le\max\{-u(\theta),-u(\tau)\}\text{,}
\]
so $u$ is quasiconcave. Continuity of $u$ follows because $u_{n}\to u$
uniformly on compact sets and each $u_{n}$ is continuous.

(ii) Fix $\epsilon>0$ and set 
\[
\mathbb{S}_{\epsilon}=\{\theta\in\mathbb{T}:\Vert\theta-\theta_{0}\Vert=\epsilon\}\text{,}
\]
and define 
\[
\delta_{\epsilon}=u(\theta_{0})-\sup_{\theta\in\mathbb{S}_{\epsilon}}u(\theta)\text{.}
\]
By the strict inequality $u(\theta)<u(\theta_{0})$ for every $\theta\neq\theta_{0}$,
we have $\delta_{\epsilon}>0$. Pointwise convergence at each $\theta\in\mathbb{T}$
gives $u_{n}(\theta_{0})\to u(\theta_{0})$, and the general bound
\[
\limsup_{n\to\infty}\sup_{\theta\in\mathbb{S}_{\epsilon}}u_{n}(\theta)\le\sup_{\theta\in\mathbb{S}_{\epsilon}}\limsup_{n\to\infty}u_{n}(\theta)=\sup_{\theta\in\mathbb{S}_{\epsilon}}u(\theta)
\]
implies that there exists $N$ such that, for all $n\ge N$, 
\begin{equation}
u_{n}(\theta_{0})-\sup_{\theta\in\mathbb{S}_{\epsilon}}u_{n}(\theta)\ge \delta_{\epsilon}/2=\alpha_{\epsilon}>0\text{.}\label{eq:alpha-gap}
\end{equation}

Now fix any $\tau\in\mathbb{T}\backslash\mathbb{B}_{\epsilon}(\theta_{0})$,
and let $\lambda_{1}=\epsilon/\Vert\tau-\theta_{0}\Vert\in(0,1]$
and $\theta_{1}=\theta_{0}+\lambda_{1}(\tau-\theta_{0})\in\mathbb{S}_{\epsilon}$.
By quasiconcavity of $u_{n}$, for all $n$, 
\[
u_{n}(\theta_{1})\ge\min\{u_{n}(\theta_{0}),u_{n}(\tau)\}\text{.}
\]
For $n\ge N$, (\ref{eq:alpha-gap}) gives $u_{n}(\theta_{0})>u_{n}(\theta_{1})$,
hence the minimum on the right cannot exceed $u_{n}(\theta_{1})$
unless $u_{n}(\tau)\le u_{n}(\theta_{1})$. Therefore, 
\[
u_{n}(\tau)\le u_{n}(\theta_{1})\le\sup_{\theta\in\mathbb{S}_{\epsilon}}u_{n}(\theta)\le u_{n}(\theta_{0})-\alpha_{\epsilon}\text{,}\qquad\text{for }n\ge N\text{.}
\]
Taking the supremum over $\tau\in\mathbb{T}\backslash\mathbb{B}_{\epsilon}(\theta_{0})$
and then the $\limsup$ as $n\to\infty$ yields 
\[
\limsup_{n\to\infty}\sup_{\theta\in\mathbb{T}\backslash\mathbb{B}_{\epsilon}(\theta_{0})}u_{n}(\theta)\le\limsup_{n\to\infty}\bigl(u_{n}(\theta_{0})-\alpha_{\epsilon}\bigr)=u(\theta_{0})-\alpha_{\epsilon}<u(\theta_{0})\text{,}
\]
as required. 
\end{proof}

When $\mathbb{T}$ is convex (as assumed in this subsection), Lemma \ref{lem:-quasiconcave} implies that B4 can be replaced with the following quasiconcavity assumption, since we can choose $\mathbb{S}=\overline{\mathbb{B}}_{\epsilon}(\theta_{0})$ for some $\epsilon>0$ with $\overline{\mathbb{B}}_{\epsilon}(\theta_{0})\subsetneq\mathbb{T}$.

\begin{description}
\item[{B4a}] For $\mathrm{P}$-a.s. $\omega\in\Omega$, and for each $n\in\mathbb{N}$, $u_{n}\left(\omega,\cdot\right):\mathbb{T}\to\mathbb{R}$ is continuous and quasiconcave.
\end{description}

\subsection{Assumptions A2a and A2b}

Next, we wish to verify that A2a and A2b together imply A2. To this
end, we require the following uniform strong law of large numbers,
which is standard and can be found, for example, in \citet[Thm. 9.60]{Shapiro2021}.
Make the following assumptions:
\begin{description}
\item[{C1}] $\mathbb{K}\subset\mathbb{T}$ is compact.
\item[{C2}] $u_{n}\left(\omega,\theta\right)=n^{-1}\sum_{i=1}^{n}v\left(X_{i}(\omega);\theta\right)$,
where $v$ is Carath\'{e}odory in the sense that $v\left(\cdot;\theta\right)$
is $\left(\mathbb{X},\mathfrak{B}_{\mathbb{X}}\right)$-measurable
for each $\theta\in\mathbb{T}$ and $v\left(x;\cdot\right)$ is continuous
for each $x\in\mathbb{X}$.
\item[{C3}] There exists a dominating function $\Delta:\mathbb{X}\rightarrow\mathbb{R}$
such that $\left|v\left(x;\theta\right)\right|\le\Delta\left(x\right)$
for each $\theta\in\mathbb{T}$, and $\mathrm{E}\left[\Delta\left(X\right)\right]<\infty$.
\end{description}
\begin{thm}
\label{thm:-ulln}Assume that $\left(X_{i}\right)_{i\in\mathbb{N}}$
are IID. If C1--C3 hold, then
\[
\sup_{\theta\in\mathbb{K}}\left|u_{n}\left(\omega,\theta\right)-u\left(\theta\right)\right|\stackrel[n\rightarrow\infty]{\mathrm{P}\text{-a.s.}}{\longrightarrow}0\text{,}
\]
where $u\left(\theta\right)=\mathrm{E}\left[v\left(X;\theta\right)\right]$.
\end{thm}
Observe that, for any compact set $\mathbb{K}\subset\mathbb{T}$, A1,
A2a and A2b together allow us to apply Theorem \ref{thm:-ulln} to conclude that \eqref{eq:-uniform-convergence-loglike}
holds $\mathrm{P}$-a.s. on $\mathbb{K}$. Since $\mathbb{T}\subset\mathbb{R}^{p}$ is either open or closed is the intersection of an open and closed set, it is a locally compact separable metric space and is hence hemicompact and we can construct a  sequence of nested compact
sets $\left(\mathbb{K}_{i}\right)_{i\in\mathbb{N}}$ such that $\mathbb{T}=\bigcup_{i\in\mathbb{N}}\mathbb{K}_{i}$.
Then, since every compact $\mathbb{K}\subset\mathbb{T}$ is contained in some $\mathbb{K}_{i}$,
A2 is implied by the intersection of the events that \eqref{eq:-uniform-convergence-loglike} holds when $\mathbb{K}=\mathbb{K}_{i}$, for each $i\in\mathbb{N}$,
which each hold $\mathrm{P}$-a.s. and thus jointly hold on a $\mathrm{P}$-a.s.
set, as required.

\subsection{Convergence of the BPIC and WBIC} \label{sec:-bpic-wbic-proofs}

Firstly, we wish to demonstrate that Proposition \ref{prop:-consistency-of-beta-posterior} is a direct consequence
of Theorem \ref{thm:-feinberg}. To see that \eqref{eq:-bic-convergence} is true, we make the substitutions
$f_{n}\left(\omega,\cdot\right)=n^{-1}\sum_{i=1}^{n}\log\text{p}\left(X_{i}\left(\omega\right)\mid\cdot\right)$,
for each fixed $\omega$, and $\text{Q}_{n}=\Pi_{n}$. Under A1 and
A2, we have that $f_{n}\stackrel[n\rightarrow\infty]{\mathrm{P}\text{-a.s.}}{\longrightarrow}\mathrm{E}\left[\log\text{p}\left(X\mid\cdot\right)\right]$
uniformly on compact sets. Then, A5 and A6 together with A1 and
A2 imply that the conditions for Theorem \ref{thm:-feinberg} are met, and we therefore
have
\begin{align}
\lim_{n\rightarrow\infty}\int_{\mathbb{T}}\frac{1}{n}\sum_{i=1}^{n}\log\text{p}\left(X_{i}\left(\omega\right)\mid\cdot\right)\Pi_{n}\left(\omega,\text{d}\theta\right)=\int_{\mathbb{T}}\mathrm{E}\left[\log\text{p}\left(X\mid\theta\right)\right]\Pi_{0}\left(\text{d}\theta\right)\text{,} \label{eq:-int-mean-log-like-convergence}
\end{align}
for $\mathrm{P}$-a.s. $\omega\in\Omega$. Thus, given expression \eqref{eq:-bpic}, we
have
\begin{align*}
\lim_{n\rightarrow\infty}\text{BPIC}_{n} & =-2\lim_{n\rightarrow\infty}\int_{\mathbb{T}}\frac{1}{n}\sum_{i=1}^{n}\log\text{p}\left(X_{i}\left(\omega\right)\mid\cdot\right)\Pi_{n}\left(\omega,\text{d}\theta\right)+2\lim_{n\rightarrow\infty}\frac{p}{n}\\
 & =-2\int_{\mathbb{T}}\mathrm{E}\left[\log\text{p}\left(X\mid\theta\right)\right]\Pi_{0}\left(\text{d}\theta\right)\text{,}
\end{align*}
for $\mathrm{P}$-a.s. $\omega\in\Omega$, as required. Result \eqref{eq:-wbic-convergence} follows
similarly, using the definition of the WBIC and taking $\text{Q}_{n}=\Pi_{n}^{\beta_{n}}$,
with $n\beta_{n}\to\infty$.

Corollary \ref{cor:-pbic-wbic-delta-convergence} then follows by the fact that A1--A5 imply that the conditions
for Proposition \ref{prop:-consistency-of-beta-posterior} and Theorem \ref{thm:-feinberg} are both satisfied, with \eqref{eq:-int-mean-log-like-convergence} satisfied,
and $\Pi_{n}^{\beta_{n}}$ converging $\mathrm{P}$-a.s.w. to $\Pi_{0}=\delta_{\theta_{0}}$.
The result is obtained by observing that 
\begin{align}
\int_{\mathbb{T}}\mathrm{E}\left[\log\text{p}\left(X\mid\theta\right)\right]\delta_{\theta_{0}}\left(\text{d}\theta\right)=\mathrm{E}\left[\log\text{p}\left(X\mid\theta_{0}\right)\right]\text{,} \label{eq:-delta-measure-limit-log-like}
\end{align}
as per \eqref{eq:-delta-measures}.

\subsection{Convergence of the DIC}

The following result is required to establish the convergence of the posterior mean, i.e. \eqref{eq:-mean-convergence}.

\begin{prop} \label{prop:-mean-convergence}
Let $\left(\mathrm{Q}_{n}\right)_{n\in\mathbb{N}}$ converge to $\mathrm{Q}$,
$\mathrm{P}$-a.s.w., and suppose that $\left\Vert \cdot\right\Vert $
is AUI with respect to $\left(\mathrm{Q}_{n}\right)_{n\in\mathbb{N}}$,
$\mathrm{P}$-a.s., for any choice of norm $\left\Vert \cdot\right\Vert $
on $\mathbb{R}^{p}$. Then, $\bar{\vartheta}_{n}:\Omega\rightarrow\mathbb{R}^{p}$
converges to $\bar{\vartheta}_{0}\in\mathbb{R}^{p}$, $\mathrm{P}$-a.s.,
where
\[
\bar{\vartheta}_{n}\left(\omega\right)=\left(\int_{\mathbb{T}}\theta_{1}\,\mathrm{Q}_{n}\left(\omega,\mathrm{d}\theta\right),\dots,\int_{\mathbb{T}}\theta_{j}\,\mathrm{Q}_{n}\left(\omega,\mathrm{d}\theta\right),\dots,\int_{\mathbb{T}}\theta_{p}\,\mathrm{Q}_{n}\left(\omega,\mathrm{d}\theta\right)\right)\text{,}
\]
for each $n\in\mathbb{N}$, and
\[
\bar{\vartheta}_{0}=\left(\int_{\mathbb{T}}\theta_{1}\,\mathrm{Q}\left(\mathrm{d}\theta\right),\dots,\int_{\mathbb{T}}\theta_{j}\,\mathrm{Q}\left(\mathrm{d}\theta\right),\dots,\int_{\mathbb{T}}\theta_{p}\,\mathrm{Q}\left(\mathrm{d}\theta\right)\right)\text{.}
\]
\end{prop}
\begin{proof}[Proof of Proposition \ref{prop:-mean-convergence}]
Consider that
\[
\left\Vert \bar{\vartheta}_{n}-\bar{\vartheta}_{0}\right\Vert _{1}=\sum_{j=1}^{p}\left|\int_{\mathbb{T}}\theta_{j}\,\mathrm{Q}_{n}\left(\mathrm{d}\theta\right)-\int_{\mathbb{T}}\theta_{j}\,\mathrm{Q}\left(\mathrm{d}\theta\right)\right|\text{,}
\]
and so $\left\Vert \bar{\vartheta}_{n}-\bar{\vartheta}_{0}\right\Vert _{1}\stackrel[n\rightarrow\infty]{\mathrm{P}\text{-a.s.}}{\longrightarrow}0$
if, for each $j\in\left[p\right]$,
\[
\left|\int_{\mathbb{T}}\theta_{j}\,\mathrm{Q}_{n}\left(\mathrm{d}\theta\right)-\int_{\mathbb{T}}\theta_{j}\,\mathrm{Q}\left(\mathrm{d}\theta\right)\right|\stackrel[n\rightarrow\infty]{\mathrm{P}\text{-a.s.}}{\longrightarrow}0\text{.}
\]
By Theorem \ref{thm:-feinberg}, this holds if, for each $j$, the map $\theta_{j}\mapsto\left|\theta_{j}\right|$
is AUI with respect to $\left(\mathrm{Q}_{n}\right)_{n\in\mathbb{N}}$,
$\mathrm{P}$-a.s., which is true if $\left\Vert \cdot\right\Vert _{1}$ is AUI
with respect to $\left(\mathrm{Q}_{n}\right)_{n\in\mathbb{N}}$, $\mathrm{P}$-a.s.,
since $\left|\theta_{j}\right|\le\left\Vert \theta\right\Vert _{1}$. The result then follows from the equivalence of norms
on $\mathbb{R}^{p}$.
\end{proof}

By \eqref{eq:-dic}, the limit of $\text{DIC}_{n}$ is determined by the limits of the form \eqref{eq:-mean-convergence-log-like} and \eqref{eq:-int-mean-log-like-convergence}. Result \eqref{eq:-mean-convergence-log-like} follows from the same argument as that of Appendix~\ref{sec:-bpic-wbic-proofs}, under assumptions A1, A2, A5, and A6. By Proposition \ref{prop:-mean-convergence}, A6 and A7 imply result \eqref{eq:-mean-convergence}, where A1 and A2 then imply \eqref{eq:-mean-convergence-log-like} by the definition of continuous convergence, as previously discussed, which establishes the first result of Proposition \ref{prop:-dic-limit}. When A6 is replaced by A3 and A4, $\bar{\vartheta}_{0}=\theta_{0}$, which thus implies the second result of the proposition by the same argument as \eqref{eq:-delta-measures}.


\section{Technical results for examples} \label{sec: Appendix-B}

\subsection{Geometric model limits}

We verify the sufficient conditions of Corollary \ref{cor:-pbic-wbic-delta-convergence}
and Proposition \ref{prop:-dic-limit} for the geometric model from
Section \ref{sec:the-geometric-model}. Recall that the model has
PMF 
\[
\text{p}\!\left(x\mid\theta\right)=\left(1-\theta\right)^{x}\theta
\]
for $x\in\mathbb{X}=\mathbb{N}\cup\left\{ 0\right\}$ and $\theta\in\mathbb{T}=\left(0,1\right)$,
and we equip $\mathbb{T}$ with the uniform prior measure $\Pi$,
characterised by density $\pi\!\left(\theta\right)=1$ for each $\theta\in\mathbb{T}$.

\paragraph{Verification of A1--A4.}

Taking logarithms yields 
\[
\log\text{p}\!\left(X\mid\theta\right)=X\log\left(1-\theta\right)+\log\theta\text{,}
\]
which is continuous in $\theta$ for each $X\in\mathbb{X}$, and is
thus Carath\'{e}odory, verifying A1. For A2a, 
\[
\text{E}\!\left[\log\text{p}\!\left(X\mid\theta\right)\right]=\text{E}X\log\left(1-\theta\right)+\log\theta
\]
is finite for each fixed $\theta\in\left(0,1\right)$ whenever $\text{E}X<\infty$.
For A2b, since every compact subset of $\mathbb{T}$ is contained
in some $\left[a,1-a\right]$ with $a\in\left(0,1/2\right)$ and both
$\log(1-\theta)$ and $\log\theta$ are continuous on $\left[a,1-a\right]$,
the Weierstrass extreme value theorem implies 
\[
\left|\log\text{p}\!\left(X\mid\theta\right)\right|\le X\sup_{\theta\in\left[a,1-a\right]}\left|\log\!\left(1-\theta\right)\right|
+\sup_{\theta\in\left[a,1-a\right]}\left|\log\theta\right|=\Delta\!\left(X\right)\text{,}
\]
with $\text{E}\!\left[\Delta\!\left(X\right)\right]<\infty$ whenever
$\text{E}X<\infty$, so A2 holds. For A3, 
\begin{equation}
\text{E}\!\left[\log\text{p}\!\left(X\mid\theta\right)\right]=\text{E}X\log\left(1-\theta\right)+\log\theta\label{eq:-000020e-00005bloglike-00005d-000020geometric}
\end{equation}
is strictly concave and is maximised at $\theta_{0}=1/\left\{ 1+\text{E}X\right\}$.
The concavity of $\theta\mapsto\log\text{p}\!\left(x\mid\theta\right)$,
for each $x\in\mathbb{X}$, implies quasiconcavity, which verifies
A4. Since $\Pi\!\left(\left(\theta_{0}-\rho,\theta_{0}+\rho\right)\right)>0$
for each $\rho>0$ whenever $\text{E}X>0$, Proposition \ref{prop:-consistency-of-beta-posterior}
implies that the power posterior sequence $\left(\Pi_{n}^{\beta_{n}}\right)_{n\in\mathbb{N}}$
converges to $\delta_{\theta_{0}}$, $\text{P-a.s.w.}$, whenever
$n\beta_{n}\to\infty$ (this provides A6 with $\Pi_{0}=\delta_{\theta_{0}}$).

\paragraph{Verification of A5 (AUI).}

Define 
\[
f_{n}\!\left(\theta\right)=\bar{X}_{n}\log\left(1-\theta\right)+\log\theta\text{,}\qquad 
f\!\left(\theta\right)=\text{E}X\log\left(1-\theta\right)+\log\theta\text{,}
\]
where $\bar{X}_{n}=n^{-1}\sum_{i=1}^{n}X_{i}$. By A2, for every compact
$\mathbb{K}\subset(0,1)$, 
\begin{equation}
\sup_{\theta\in\mathbb{K}}\left|f_{n}\!\left(\theta\right)-f\!\left(\theta\right)\right|\xrightarrow[n\to\infty]{\text{P-a.s.}}0\text{,}\label{eq:-unif-on-K-geom}
\end{equation}
By A3 and the boundary behaviour $f(\theta)\to-\infty$ as $\theta\searrow0$
or $\theta\nearrow1$, there exist $0<a<\theta_{0}<b<1$ and $\varepsilon>0$
such that 
\begin{equation}
\sup_{\theta\in(0,a]\cup[b,1)}f(\theta)\le f(\theta_{0})-3\varepsilon\text{.}\label{eq:-separation-geom}
\end{equation}
Let $\mathbb{K}=[a,b]$. Using \eqref{eq:-unif-on-K-geom} on $\mathbb{K}$, and noting that both $f$ and $f_{n}$ are concave in $\theta$ (so their suprema on $\mathbb{K}^{\text{c}}=(0,a]\cup[b,1)$ are attained at the boundary points $a$ or $b$), there exists $N_{1}$ such that for all $n\ge N_{1}$, 
\begin{equation}
\sup_{\theta\in\mathbb{K}^{\text{c}}}f_{n}(\theta)\le f(\theta_{0})-2\varepsilon\text{,}\qquad
\inf_{\theta\in\mathbb{K}}f_{n}(\theta)\ge f(\theta_{0})-\varepsilon\text{.}\label{eq:-unif-sep-geom}
\end{equation}
Writing the normalising constant of $\Pi_{n}^{\beta_{n}}$ as 
\[
Z_{n}=\int_{0}^{1}\exp\!\left\{ n\beta_{n}f_{n}(\theta)\right\} \text{d}\theta\text{,}
\]
\eqref{eq:-unif-sep-geom} yields, for $n\ge N_{1}$, 
\begin{equation}
Z_{n}\ge\int_{\mathbb{K}}\exp\!\left\{ n\beta_{n}f_{n}(\theta)\right\} \text{d}\theta
\ge\text{Leb}\!\left(\mathbb{K}\right)\exp\!\left\{ n\beta_{n}\left(f(\theta_{0})-\varepsilon\right)\right\}\text{,}\label{eq:-Z-lb-geom}
\end{equation}
where $\text{Leb}$ denotes Lebesgue measure on the relevant space.
Consequently, for any measurable $\mathbb{A}\subset(0,1)$, 
\begin{equation}
\Pi_{n}^{\beta_{n}}(\mathbb{A})=\frac{\int_{\mathbb{A}}\exp\{n\beta_{n}f_{n}(\theta)\}\text{d}\theta}{Z_{n}}
\le\frac{\text{Leb}\!\left(\mathbb{A}\right)\exp\{n\beta_{n}\sup_{\theta\in\mathbb{A}}f_{n}(\theta)\}}
{\text{Leb}\!\left(\mathbb{K}\right)\exp\{n\beta_{n}(f(\theta_{0})-\varepsilon)\}}\text{.}\label{eq:-tail-template-geom}
\end{equation}
Applying \eqref{eq:-unif-sep-geom} in \eqref{eq:-tail-template-geom}
with $\mathbb{A}=\mathbb{K}^{\text{c}}$ gives the exponential tail
bound 
\[
\Pi_{n}^{\beta_{n}}\!\left(\mathbb{K}^{\text{c}}\right)\le 
C\exp\!\left\{ -n\beta_{n}\varepsilon\right\}\text{,}\qquad n\ge N_{1}\text{,}
\]
with the explicit constant 
$C=\text{Leb}\!\left(\mathbb{K}^{\text{c}}\right)/\text{Leb}\!\left(\mathbb{K}\right)$.

We now verify A5 (AUI) for the sequence of average log-likelihoods \(f_{n}\). Fix \(\delta>0\) and split
\[
\int_{\mathbb{T}}\left|f_{n}(\theta)\right|\mathbf{1}_{\{\left|f_{n}(\theta)\right|\ge\delta\}}\Pi_{n}^{\beta_{n}}(\text{d}\theta)
=\int_{\mathbb{K}}\left|f_{n}\right|\mathbf{1}_{\{\left|f_{n}\right|\ge\delta\}}\text{d}\Pi_{n}^{\beta_{n}}
+\int_{\mathbb{K}^{\text{c}}}\left|f_{n}\right|\mathbf{1}_{\{\left|f_{n}\right|\ge\delta\}}\text{d}\Pi_{n}^{\beta_{n}}\text{,}
\]
where \(\mathbb{K}=[a,b]\) is as above. By \eqref{eq:-unif-on-K-geom},
\[
\sup_{\theta\in\mathbb{K}}\left|f_{n}(\theta)\right|
\xrightarrow[n\to\infty]{\text{P-a.s.}}
\sup_{\theta\in\mathbb{K}}\left|f(\theta)\right|
<\infty\text{.}
\]
Hence, if \(\delta>\sup_{\theta\in\mathbb{K}}\left|f(\theta)\right|\), then there exists \(N_{2}\) (on the same \(\text{P}\)-a.s. event) such that for all \(n\ge N_{2}\),
\(\sup_{\theta\in\mathbb{K}}\left|f_{n}(\theta)\right|<\delta\); consequently the first integral equals \(0\) for all \(n\ge N_{2}\).

For the second integral, the separation in \eqref{eq:-unif-sep-geom} and the normalising constant bound \eqref{eq:-Z-lb-geom} imply that for \(n\ge N_{1}\),
\[
\pi_{n}^{\beta_{n}}(\theta)=\frac{\exp\{n\beta_{n}f_{n}(\theta)\}}{Z_{n}}
\le \frac{e^{-n\beta_{n}\varepsilon}}{\text{Leb}(\mathbb{K})}\text{,}\qquad \theta\in\mathbb{K}^{\text{c}},
\]
so that for any nonnegative \(g\),
\[
\int_{\mathbb{K}^{\text{c}}} g(\theta)\,\Pi_{n}^{\beta_{n}}(\text{d}\theta)\le \frac{e^{-n\beta_{n}\varepsilon}}{\text{Leb}(\mathbb{K})}\int_{\mathbb{K}^{\text{c}}} g(\theta)\,\text{d}\theta\text{.}
\]
Taking \(g(\theta)=\left|f_{n}(\theta)\right|\mathbf{1}_{\{\left|f_{n}(\theta)\right|\ge\delta\}}\) yields
\[
\int_{\mathbb{K}^{\text{c}}}\left|f_{n}\right|\mathbf{1}_{\{\left|f_{n}\right|\ge\delta\}}\text{d}\Pi_{n}^{\beta_{n}}
\le \frac{e^{-n\beta_{n}\varepsilon}}{\text{Leb}(\mathbb{K})}\int_{\mathbb{K}^{\text{c}}}\left|f_{n}(\theta)\right|\text{d}\theta\text{.}
\]
The integral on the right is finite (indeed, a.s. eventually uniformly bounded) because
\[
\left|f_{n}(\theta)\right|\le \left|\bar X_{n}\right|\left|\log(1-\theta)\right|+\left|\log\theta\right|
\]
and the boundary singularities are Lebesgue-integrable, i.e.,
\[
\int_{0}^{a}\left|\log\theta\right|\text{d}\theta=a\left(1-\log a\right)<\infty\text{,}\qquad
\int_{b}^{1}\left|\log(1-\theta)\right|\text{d}\theta=(1-b)\left(1-\log(1-b)\right)<\infty\text{.}
\]
Moreover, since \(\text{E}X<\infty\) implies \(\bar X_{n}\to \text{E}X\), \(\text{P-a.s.}\), there exists a deterministic constant \(M>0\) (e.g. \(M=\left|\text{E}X\right|+1\)) such that, \(\text{P-a.s.}\) and for all sufficiently large \(n\),
\[
\int_{\mathbb{K}^{\text{c}}}\left|f_{n}(\theta)\right|\text{d}\theta
\le
M\int_{\mathbb{K}^{\text{c}}}\left|\log(1-\theta)\right|\text{d}\theta
+\int_{\mathbb{K}^{\text{c}}}\left|\log\theta\right|\text{d}\theta
<\infty\text{.}
\]
Therefore, for all large \(n\),
\[
\int_{\mathbb{K}^{\text{c}}\!}\left|f_{n}\right|\mathbf{1}_{\{\left|f_{n}\right|\ge\delta\}}\mathrm{d}\Pi_{n}^{\beta_{n}}
\le \frac{e^{-n\beta_{n}\varepsilon}}{\mathrm{Leb}(\mathbb{K})}
   \int_{\mathbb{K}^{\text{c}}\!}\!\left|f_{n}(\theta)\right|\mathrm{d}\theta
\stackrel[\;n\to\infty\;]{\text{P-a.s.}}{\longrightarrow} 0\text{.}
\]
Combining the two parts, taking \(\limsup_{n\to\infty}\) and then letting \(\delta\to\infty\) establishes AUI:
\[
\lim_{\delta\to\infty}\limsup_{n\to\infty}\int_{\mathbb{T}}\left|f_{n}(\theta)\right|\mathbf{1}_{\{\left|f_{n}(\theta)\right|\ge\delta\}}\Pi_{n}^{\beta_{n}}(\text{d}\theta)=0\text{,}\qquad \text{P-a.s.}
\]

\paragraph{Limits for BPIC and WBIC.}

By A2 (uniform convergence on compacta and hemi-compactness of $\mathbb{T}$),
$\left(f_{n}\right)_{n\in\mathbb{N}}$ converges continuously to $f$,
$\text{P-a.s.}$; by the above, A5 holds; and by Proposition \ref{prop:-consistency-of-beta-posterior},
A6 holds with $\Pi_{0}=\delta_{\theta_{0}}$. Therefore Theorem \ref{thm:-feinberg}
applies and yields 
\[
\int_{\mathbb{T}}f_{n}(\theta)\,\Pi_{n}^{\beta_{n}}(\text{d}\theta)\longrightarrow\int_{\mathbb{T}}f(\theta)\,\delta_{\theta_{0}}(\text{d}\theta)=f(\theta_{0})=\text{E}\!\left[\log\text{p}(X\mid\theta_{0})\right]\text{,}\qquad\text{P-a.s.}
\]
Hence, as $n\to\infty$ with $n\beta_{n}\to\infty$, 
\[
\text{BPIC}_{n},\text{WBIC}_{n}\stackrel[n\to\infty]{\text{P-a.s.}}{\longrightarrow}-2\,\text{E}\!\left[\log\text{p}\!\left(X\mid\theta_{0}\right)\right]
=2\log\!\left\{ 1+\text{E}X\right\}-2\,\text{E}X\log\!\left\{ \frac{\text{E}X}{1+\text{E}X}\right\}\text{.}
\]

\subsection{Normal model limits}

We verify the sufficient conditions of Corollary \ref{cor:-pbic-wbic-delta-convergence}
and Proposition \ref{prop:-dic-limit} for the normal model in Section
\ref{sec:-normal-limit}. Recall that 
\[
\text{p}\left(x\mid\theta\right)=\left(2\pi\right)^{-p/2}\exp\left\{ -\frac{1}{2}\left\Vert x-\theta\right\Vert ^{2}\right\} \text{,}\qquad x\in\mathbb{X}=\mathbb{R}^{p},\ \theta\in\mathbb{T}=\mathbb{R}^{p}\text{,}
\]
and equip $\mathbb{T}$ with the normal prior $\Pi$ with PDF 
\[
\pi(\theta)=\left(2\pi\right)^{-p/2}\exp\left\{ -\frac{1}{2}\left\Vert \theta-\mu\right\Vert ^{2}\right\} \text{,}\qquad\mu\in\mathbb{R}^{p}\text{.}
\]
For $\beta_{n}>0$, the power posterior has density 
\begin{equation}
\pi_{n}^{\beta_{n}}(\theta)=\frac{\exp\{n\beta_{n}f_{n}(\theta)\}\,\pi(\theta)}{Z_{n}}\text{,}\qquad Z_{n}=\int_{\mathbb{R}^{p}}\exp\{n\beta_{n}f_{n}(\theta)\}\,\pi(\theta)\,\text{d}\theta\text{,}\label{eq:-post-normal-with-prior}
\end{equation}
where 
\[
f_{n}(\theta)=-\frac{p}{2}\log(2\pi)-\frac{1}{2n}\sum_{i=1}^{n}\left\Vert X_{i}-\theta\right\Vert ^{2}\text{.}
\]

\paragraph{Verification of A1--A4.}

Taking logarithms, 
\[
\log\text{p}(X\mid\theta)=-\frac{p}{2}\log(2\pi)-\frac{1}{2}\left\Vert X-\theta\right\Vert ^{2}=-\frac{p}{2}\log(2\pi)-\frac{1}{2}\left\{ \left\Vert X\right\Vert ^{2}-2X^{\top}\theta+\left\Vert \theta\right\Vert ^{2}\right\} \text{,}
\]
which is continuous in $\theta$ for each $X\in\mathbb{X}$ and jointly
measurable in $(X,\theta)$; hence it is Carath\'{e}odory, verifying A1.
For A2a, for each fixed $\theta$, 
\[
\text{E}\left[\log\text{p}(X\mid\theta)\right]=-\frac{p}{2}\log(2\pi)-\frac{1}{2}\left\{ \text{E}\left[\left\Vert X\right\Vert ^{2}\right]-2\text{E}X^{\top}\theta+\left\Vert \theta\right\Vert ^{2}\right\} 
\]
is finite whenever $\text{E}\left[\left\Vert X\right\Vert ^{2}\right]<\infty$.
For A2b, if $\mathbb{K}\subset\mathbb{T}$ is compact, then $\sup_{\theta\in\mathbb{K}}\left\Vert \theta\right\Vert =R<\infty$,
and 
\[
\bigl|\log\text{p}(X\mid\theta)\bigr|\le\frac{p}{2}\log(2\pi)+\frac{1}{2}\left\Vert X\right\Vert ^{2}+R\left\Vert X\right\Vert +\frac{1}{2}R^{2}=\Delta(X)\text{,}
\]
with $\text{E}\Delta(X)<\infty$ if $\text{E}\left\Vert X\right\Vert ^{2}<\infty$.
Hence A2 holds. For A3, writing $\theta_{0}=\text{E}X$ and completing
the square, 
\[
\text{E}\left[\log\text{p}(X\mid\theta)\right]=-\frac{p}{2}\log(2\pi)-\frac{1}{2}\left\{ \text{E}\left[\left\Vert X-\theta_{0}\right\Vert ^{2}\right]+\left\Vert \theta-\theta_{0}\right\Vert ^{2}\right\} \text{,}
\]
which is strictly concave in $\theta$ and uniquely maximised at $\theta_{0}$.
Concavity implies quasiconcavity, verifying A4 via B4a. Since the normal prior
assigns positive mass to all open balls, $\Pi\left(\mathbb{B}_{\rho}(\theta_{0})\right)>0$
for every $\rho>0$. By Proposition \ref{prop:-consistency-of-beta-posterior},
if $n\beta_{n}\to\infty$ then $\Pi_{n}^{\beta_{n}}\Rightarrow\delta_{\theta_{0}}$
$\text{P}$-a.s.w. (this verifies A6 with $\Pi_{0}=\delta_{\theta_{0}}$).
We note explicitly that this step uses only that $\pi$ is strictly
positive in a neighbourhood of $\theta_{0}$.

\paragraph{Verification of A5 (AUI).}

Define also 
\[
f(\theta)=-\frac{p}{2}\log(2\pi)-\frac{1}{2}\text{E}\left[\left\Vert X-\theta\right\Vert ^{2}\right]\text{.}
\]
By A2, for every compact $\mathbb{K}\subset\mathbb{T}$, 
\begin{equation}
\sup_{\theta\in\mathbb{K}}\left|f_{n}(\theta)-f(\theta)\right|\xrightarrow[n\to\infty]{\text{P-a.s.}}0\text{.}\label{eq:-unif-on-K-normal}
\end{equation}
Fix $r>0$ and set $\mathbb{K}=\bar{\mathbb{B}}_{r}(\theta_{0})$.
Since 
\[
f(\theta)=f(\theta_{0})-\frac{1}{2}\left\Vert \theta-\theta_{0}\right\Vert ^{2}\text{,}
\]
we have the separation 
\begin{equation}
\sup_{\theta\in\mathbb{K}^{\text{c}}}f(\theta)\le f(\theta_{0})-\frac{1}{2}r^{2}\text{,}\qquad\inf_{\theta\in\mathbb{K}}f(\theta)=f(\theta_{0})-\frac{1}{2}r^{2}\text{.}\label{eq:-sep-normal}
\end{equation}
Combining (\ref{eq:-unif-on-K-normal}) with (\ref{eq:-sep-normal}),
there exists $N_{1}$ (on a $\text{P}$-a.s. set) such that for all
$n\ge N_{1}$, 
\begin{equation}
\sup_{\theta\in\mathbb{K}^{\text{c}}}f_{n}(\theta)\le f(\theta_{0})-\frac{1}{4}r^{2}\text{,}\qquad\inf_{\theta\in\mathbb{K}}f_{n}(\theta)\ge f(\theta_{0})-\frac{3}{4}r^{2}\text{.}\label{eq:-unif-sep-normal}
\end{equation}
Let 
\[
Z_{n}=\int_{\mathbb{R}^{p}}\exp\left\{ n\beta_{n}f_{n}(\theta)\right\} \pi(\theta)\text{d}\theta\text{.}
\]
Using (\ref{eq:-unif-sep-normal}) and the fact that $\inf_{\theta\in\mathbb{K}}\pi(\theta)>0$,
we obtain, for $n\ge N_{1}$, 
\begin{equation}
Z_{n}\ge\inf_{\theta\in\mathbb{K}}\pi(\theta)\int_{\mathbb{K}}\exp\left\{ n\beta_{n}f_{n}(\theta)\right\} \text{d}\theta\ge\mathrm{Leb}(\mathbb{K})\inf_{\theta\in\mathbb{K}}\pi(\theta)\exp\left\{ n\beta_{n}\left(f(\theta_{0})-\frac{3}{4}r^{2}\right)\right\} \text{.}\label{eq:-Z-lb-normal-prior}
\end{equation}
To control the tails, for $\Delta>0$ define the annuli 
\[
\mathbb{A}_{m}=\left\{ \theta\in\mathbb{R}^{p}:r+m\Delta\le\left\Vert \theta-\theta_{0}\right\Vert <r+(m+1)\Delta\right\} \text{,}\qquad m\in\mathbb{N}\text{.}
\]
On $\mathbb{A}_{m}$, (\ref{eq:-unif-sep-normal}) yields 
\[
\sup_{\theta\in\mathbb{A}_{m}}f_{n}(\theta)\le f(\theta_{0})-\frac{1}{2}\bigl(r+m\Delta\bigr)^{2}+\frac{1}{4}r^{2}=f(\theta_{0})-\frac{1}{2}\bigl(m^{2}\Delta^{2}+2mr\Delta+\frac{1}{2}r^{2}\bigr)\text{.}
\]
Since $\sup_{\theta\in\mathbb{R}^{p}}\pi(\theta)=(2\pi)^{-p/2}$,
it follows from (\ref{eq:-post-normal-with-prior}) and (\ref{eq:-Z-lb-normal-prior})
that, for $n\ge N_{1}$, 
\begin{equation}
\Pi_{n}^{\beta_{n}}(\mathbb{A}_{m})=\frac{\int_{\mathbb{A}_{m}}\exp\{n\beta_{n}f_{n}(\theta)\}\pi(\theta)\,\text{d}\theta}{Z_{n}}\le\frac{(2\pi)^{-p/2}\mathrm{Leb}(\mathbb{A}_{m})}{\mathrm{Leb}(\mathbb{K})\inf_{\theta\in\mathbb{K}}\pi(\theta)}\exp\left\{ -\frac{n\beta_{n}}{2}\bigl(m^{2}\Delta^{2}+2mr\Delta+\frac{1}{2}r^{2}\bigr)\right\} .\label{eq:-annulus-tail-normal-prior}
\end{equation}
Using $\mathrm{Leb}(\mathbb{A}_{m})\le C'_{p}(r+m\Delta)^{p-1}\Delta$, where $C'_{p}$ is a dimension-dependent constant from the unit-ball volume formula,
we deduce 
\[
\Pi_{n}^{\beta_{n}}(\mathbb{A}_{m})\le\frac{C_{p}'}{\mathrm{Leb}(\mathbb{K})\inf_{\theta\in\mathbb{K}}\pi(\theta)}(r+m\Delta)^{p-1}\Delta\exp\left\{ -\frac{n\beta_{n}}{2}\bigl(m^{2}\Delta^{2}+2mr\Delta+\frac{1}{2}r^{2}\bigr)\right\} .
\]

We now show AUI for $\left(f_{n}\right)_{n\in\mathbb{N}}$ with respect
to $\left(\Pi_{n}^{\beta_{n}}\right)_{n\in\mathbb{N}}$. Fix $\delta>0$
and split 
\[
\int_{\mathbb{R}^{p}}\left|f_{n}(\theta)\right|\mathbf{1}_{\{\left|f_{n}(\theta)\right|\ge\delta\}}\Pi_{n}^{\beta_{n}}(\text{d}\theta)=\int_{\mathbb{K}}\left|f_{n}\right|\mathbf{1}_{\{\left|f_{n}\right|\ge\delta\}}\text{d}\Pi_{n}^{\beta_{n}}+\sum_{m=0}^{\infty}\int_{\mathbb{A}_{m}}\left|f_{n}\right|\mathbf{1}_{\{\left|f_{n}\right|\ge\delta\}}\text{d}\Pi_{n}^{\beta_{n}}\text{.}
\]
By (\ref{eq:-unif-on-K-normal}), $\underset{n\to\infty}{\overset{\mathrm{P\text{-}a.s.}}{\longrightarrow}}\sup_{\theta\in\mathbb{K}}\left|f(\theta)\right|<\infty$;
hence if $\delta>\sup_{\theta\in\mathbb{K}}\left|f(\theta)\right|$,
the first integral is $0$ for all large $n$ on the same $\mathrm{P}$-a.s.
event. For the tail sum, $f_{n}$ is a quadratic polynomial in $\theta$
with coefficients that are $\mathrm{P}$-a.s. eventually bounded (by
the strong laws for $\bar{X}_{n}$ and $n^{-1}\sum_{i=1}^{n}\left\Vert X_{i}\right\Vert ^{2}$).
Therefore there exists $C>0$ such that, for all large $n$ and all
$m\ge0$, 
\[
\sup_{\theta\in\mathbb{A}_{m}}\left|f_{n}(\theta)\right|\le C\bigl(1+m^{2}\Delta^{2}\bigr)\text{,}
\]
and hence, by (\ref{eq:-annulus-tail-normal-prior}), there is $C'>0$ such that
\[
\int_{\mathbb{A}_{m}}\left|f_{n}\right|\mathbf{1}_{\{\left|f_{n}\right|\ge\delta\}}\text{d}\Pi_{n}^{\beta_{n}}
\le\frac{C'\bigl(1+m^{2}\Delta^{2}\bigr)(r+m\Delta)^{p-1}\Delta}{\mathrm{Leb}(\mathbb{K})\inf_{\theta\in\mathbb{K}}\pi(\theta)}
\exp\left\{ -\frac{n\beta_{n}}{2}\bigl(m^{2}\Delta^{2}+2mr\Delta+\frac{1}{2}r^{2}\bigr)\right\}.
\]
Observe that $1+m^{2}\Delta^{2}\le\max\{1,\Delta^{2}\}(1+m^{2})$ and $m^{2}\Delta^{2}+2mr\Delta+(1/2)r^{2}\ge m^{2}\Delta^{2}$. Absorbing the fixed positive multiplicative factors depending only on $\Delta$, $r$, $\mathrm{Leb}(\mathbb{K})$, and $\inf_{\theta\in\mathbb{K}}\pi(\theta)$, we obtain
\[
\int_{\mathbb{A}_{m}}\left|f_{n}\right|\mathbf{1}_{\{\left|f_{n}\right|\ge\delta\}}\text{d}\Pi_{n}^{\beta_{n}}
\le C_{1}\,(1+m^{2})(r+m\Delta)^{p-1}\exp\left\{-c\,n\beta_{n}\,m^{2}\Delta^{2}\right\}\text{,}
\]
for some $c>0$ and $C_{1}>0$ independent of $m$ and $n$. Summing over $m$ and comparing with the integral of a polynomial times a Gaussian tail yields
\[
\sum_{m=0}^{\infty}(1+m^{2})(r+m\Delta)^{p-1}\exp\left\{-c\,n\beta_{n}\,m^{2}\Delta^{2}\right\}
\le\frac{C_{p,r,\Delta}}{(n\beta_{n})^{(p+2)/2}}\xrightarrow[n\to\infty]{}0\text{,}
\]
for constant $C_{p,r,\Delta}>0$. Therefore 
\[
\lim_{\delta\to\infty}\limsup_{n\to\infty}\int_{\mathbb{R}^{p}}\left|f_{n}(\theta)\right|\mathbf{1}_{\{\left|f_{n}(\theta)\right|\ge\delta\}}\Pi_{n}^{\beta_{n}}(\text{d}\theta)=0\text{,}\qquad\mathrm{P\text{-}a.s.,}
\]
which verifies the AUI condition (A5).

\paragraph{Limits for BPIC and WBIC.}

By A2 
$\left(f_{n}\right)_{n\in\mathbb{N}}$ converges continuously to $f$,
$\text{P-a.s.}$; by the above, A5 holds; and by Proposition \ref{prop:-consistency-of-beta-posterior},
A6 holds with $\Pi_{0}=\delta_{\theta_{0}}$. Therefore Theorem \ref{thm:-feinberg}
applies to the integrals with respect to $\Pi_{n}^{\beta_{n}}$ and
yields 
\[
\int_{\mathbb{T}}f_{n}(\theta)\,\Pi_{n}^{\beta_{n}}(\text{d}\theta)\xrightarrow[n\to\infty]{\text{P-a.s.}}f(\theta_{0})=-\frac{p}{2}\log(2\pi)-\frac{1}{2}\text{E}\left[\left\Vert X-\theta_{0}\right\Vert ^{2}\right]\text{.}
\]
Hence, as $n\to\infty$ with $n\beta_{n}\to\infty$, 
\[
\text{BPIC}_{n},\text{WBIC}_{n}\xrightarrow[n\to\infty]{\text{P-a.s.}}p\log(2\pi)+\text{E}\left[\left\Vert X\right\Vert ^{2}\right]-\left\Vert \text{E}X\right\Vert ^{2}\text{.}
\]

\paragraph{Limit for DIC.}

For the second term of (\ref{eq:-dic}), we verify A7 for $\left\Vert \cdot\right\Vert $
with respect to the standard posterior $\Pi_{n}\equiv\Pi_{n}^{\beta_{n}}$
with $\beta_{n}=1$, whose density is proportional to $\exp\{\sum_{i=1}^{n}\log\text{p}(X_{i}\mid\theta)\}\pi(\theta)$.
Repeating the annulus argument with $g(\theta)=\left\Vert \theta\right\Vert $, we obtain, for $n$
large, 
\[
\int_{\mathbb{R}^{p}}\left\Vert \theta\right\Vert \mathbf{1}_{\{\left\Vert \theta\right\Vert \ge\delta\}}\Pi_{n}(\text{d}\theta)\le\frac{C}{\inf_{\theta\in\mathbb{K}}\pi(\theta)}\sum_{m=0}^{\infty}(r+m\Delta)\,(r+m\Delta)^{p-1}\Delta e^{-cn(m^{2}\Delta^{2})}\xrightarrow[n\to\infty]{\text{P-a.s.}}0\text{,}
\]
establishing A7. Then Proposition \ref{prop:-mean-convergence} together
with A1, A2 and A6 gives $\bar{\theta}_{n}\xrightarrow{\text{P-a.s.}}\theta_{0}$
and hence 
\[
\frac{2}{n}\sum_{i=1}^{n}\log\text{p}(X_{i}\mid\bar{\theta}_{n})\xrightarrow[n\to\infty]{\text{P-a.s.}}2\text{E}\left[\log\text{p}(X\mid\theta_{0})\right]\text{;}
\]
cf. (\ref{eq:-mean-convergence-log-like}). Combining with the first
term yields 
\[
\text{DIC}_{n}\xrightarrow[n\to\infty]{\text{P-a.s.}}p\log(2\pi)+\text{E}\left[\left\Vert X\right\Vert ^{2}\right]-\left\Vert \text{E}X\right\Vert ^{2}\text{,}
\]
as claimed. 
We conclude with the following lemma for obtaining closed forms of the WBIC for the normal models.

\begin{lem}
\label{lem:pp-normal} Let $\mathrm{p}$ and $\pi$ be as given in
Section \ref{sec:-normal-limit}. For $\beta_{n}>0$, the power posterior 
\[
\Pi_{n}^{\beta_{n}}(\mathrm{d}\theta)\propto\biggl\{\prod_{i=1}^{n}\mathrm{p}(X_{i}\mid\theta)\biggr\}^{\beta_{n}}\Pi(\mathrm{d}\theta)\text{,}
\]
has law $\mathrm{N}(m_{n},v_{n}\mathbf{I})$, with 
\[
m_{n}=\frac{n\beta_{n}\,\bar{X}_{n}+\mu}{n\beta_{n}+1}\text{,}\qquad v_{n}=\frac{1}{n\beta_{n}+1}\text{.}
\]
\end{lem}

\begin{proof}
Write the prior density as 
\[
\pi(\theta)\propto\exp\left(-\frac{1}{2}\lVert\theta-\mu\rVert^{2}\right)\text{,}
\]
and the power likelihood as 
\[
\prod_{i=1}^{n}\mathrm{p}\left(X_{i}\mid\theta\right)^{\beta_{n}}\propto\exp\Bigl(-\frac{\beta_{n}}{2}\sum_{i=1}^{n}\lVert X_{i}-\theta\rVert^{2}\Bigr)\text{.}
\]
Hence the power posterior kernel is 
\[
\pi(\theta)\prod_{i=1}^{n}\mathrm{p}\left(X_{i}\mid\theta\right)^{\beta_{n}}\propto\exp\Bigl(-\frac{1}{2}\lVert\theta-\mu\rVert^{2}-\frac{\beta_{n}}{2}\sum_{i=1}^{n}\lVert X_{i}-\theta\rVert^{2}\Bigr)\text{.}
\]
Use the identity 
\[
\sum_{i=1}^{n}\lVert X_{i}-\theta\rVert^{2}=n\lVert\theta-\bar{X}_{n}\rVert^{2}+\sum_{i=1}^{n}\lVert X_{i}-\bar{X}_{n}\rVert^{2}\text{,}
\]
where the second term does not depend on $\theta$ and can be absorbed
into the normalising constant. Collecting the terms depending on $\theta$
gives 
\[
-\frac{1}{2}\Bigl[\lVert\theta-\mu\rVert^{2}+n\beta_{n}\lVert\theta-\bar{X}_{n}\rVert^{2}\Bigr]=-\frac{1}{2}\Bigl[(1+n\beta_{n})\lVert\theta\rVert^{2}-2(n\beta_{n}\bar{X}_{n}+\mu)^{\top}\theta\Bigr]+\mathrm{const.}
\]
Completing the square yields
\[
(1+n\beta_{n})\lVert\theta\rVert^{2}-2(n\beta_{n}\bar{X}_{n}+\mu)^{\top}\theta=(1+n\beta_{n})\lVert\theta-m_{n}\rVert^{2}-\frac{\lVert n\beta_{n}\bar{X}_{n}+\mu\rVert^{2}}{1+n\beta_{n}}\text{,}
\]
with 
\[
m_{n}=\frac{n\beta_{n}\bar{X}_{n}+\mu}{n\beta_{n}+1}\text{.}
\]
Therefore the kernel is proportional to $\exp\bigl(-2^{-1}(1+n\beta_{n})\lVert\theta-m_{n}\rVert^{2}\bigr)$,
which is the kernel of the $\mathrm{N}(m_{n},v_{n}\mathbf{I})$ density, with
$v_{n}=(n\beta_{n}+1)^{-1}$, as required. 
\end{proof}

\subsection{Laplace model limits} \label{sec:-laplace-model-limits}

We verify the sufficient conditions of Corollary \ref{cor:-pbic-wbic-delta-convergence}
and Proposition \ref{prop:-dic-limit} for the Laplace model. Recall
that 
\[
\text{p}\left(x\mid\theta\right)=\frac{1}{2\gamma}\exp\left\{ -\left|x-\mu\right|/\gamma\right\} \text{,}\qquad x\in\mathbb{X}=\mathbb{R}\text{,}\ \theta=(\mu,\gamma)\in\mathbb{T}=[-m,m]\times[s^{-1},s]\text{,}
\]
with $m>0$ and $s>1$. We equip $\mathbb{T}$ with a prior $\Pi$
whose density $\pi$ is strictly positive on $\mathbb{T}$. For $\beta_{n}>0$,
the power posterior has density 
\begin{equation}
\pi_{n}^{\beta_{n}}(\theta)=Z_{n}^{-1}\exp\{n\beta_{n}f_{n}(\theta)\}\,\pi(\theta)\text{,}\qquad Z_{n}=\int_{\mathbb{T}}\exp\{n\beta_{n}f_{n}(\theta)\}\pi(\theta)\text{d}\theta\text{,}\label{eq:-post-laplace-with-prior}
\end{equation}
where 
\[
f_{n}(\theta)=-\log(2\gamma)-\gamma^{-1}\,n^{-1}\sum_{i=1}^{n}\left|X_{i}-\mu\right|\text{.}
\]
Since $\pi$ is continuous and strictly positive on the compact set
$\mathbb{T}$, there exist constants $0<\underline{\pi}\le\overline{\pi}<\infty$
with $\underline{\pi}\le\pi(\theta)\le\overline{\pi}$ for all $\theta\in\mathbb{T}$.
Consequently, 
\[
\underline{\pi}\int_{\mathbb{T}}\exp\{n\beta_{n}f_{n}(\theta)\}\text{d}\theta\le Z_{n}\le\overline{\pi}\int_{\mathbb{T}}\exp\{n\beta_{n}f_{n}(\theta)\}\text{d}\theta\text{,}
\]
and, for any measurable $\mathbb{A}\subset\mathbb{T}$, 
\[
\frac{\underline{\pi}}{\overline{\pi}}\frac{\int_{\mathbb{A}}\exp\{n\beta_{n}f_{n}(\theta)\}\text{d}\theta}{\int_{\mathbb{T}}\exp\{n\beta_{n}f_{n}(\theta)\}\text{d}\theta}\le\Pi_{n}^{\beta_{n}}(\mathbb{A})\le\frac{\overline{\pi}}{\underline{\pi}}\frac{\int_{\mathbb{A}}\exp\{n\beta_{n}f_{n}(\theta)\}\text{d}\theta}{\int_{\mathbb{T}}\exp\{n\beta_{n}f_{n}(\theta)\}\text{d}\theta}\text{.}
\]
These bounds will be used implicitly. In particular, the presence
of the prior only affects fixed multiplicative constants in the estimates
below.

\paragraph{Verification of A1--A4.}

Taking logarithms, 
\[
\log\text{p}(X\mid\theta)=-\log(2\gamma)-\gamma^{-1}\left|X-\mu\right|\text{,}
\]
which is continuous in $\theta$ for each $X\in\mathbb{X}$ and jointly
measurable in $(X,\theta)$; hence it is Carath\'{e}odory, verifying A1.
For A2a, for each fixed $\theta$, 
\[
\text{E}\left[\log\text{p}(X\mid\theta)\right]=-\log(2\gamma)-\gamma^{-1}\text{E}\left|X-\mu\right|
\]
is finite whenever $\text{E}\left|X\right|<\infty$. For A2b, since
$\mathbb{T}$ is compact, 
\[
\bigl|\log\text{p}(X\mid\theta)\bigr|\le\sup_{\gamma\in[s^{-1},s]}\bigl|\log(2\gamma)\bigr|+s\,\left|X\right|+\sup_{(\mu,\gamma)\in\mathbb{T}}\left|\mu\right|/\gamma=\Delta(X)\text{,}
\]
with $\text{E}\Delta(X)<\infty$ if $\text{E}\left|X\right|<\infty$.
Hence A2 holds. For A3, set 
\[
f(\theta)=-\log(2\gamma)-\gamma^{-1}\,\text{E}\left|X-\mu\right|\text{.}
\]
If $X$ has a continuous distribution function, then for each fixed
$\gamma$ the map $\mu\mapsto\text{E}\left|X-\mu\right|$ is uniquely
minimized at $\mu_{0}=\text{Med}(X)$, and with $\mu=\mu_{0}$ fixed
the map $\gamma\mapsto-\log(2\gamma)-\gamma^{-1}\text{E}\left|X-\mu_{0}\right|$
is maximized at $\gamma_{0}=\text{E}\left|X-\mu_{0}\right|$ by Fermat's
condition. Assuming $\theta_{0}=(\mu_{0},\gamma_{0})\in\mathbb{T}$,
$f$ is uniquely maximized at $\theta_{0}$. Continuity on the compact
$\mathbb{T}$ then implies A4, directly. Since $\pi$ is strictly
positive on $\mathbb{T}$, in particular $\Pi(\mathbb{B}_{\rho}(\theta_{0}))>0$
for all $\rho>0$; by Proposition \ref{prop:-consistency-of-beta-posterior},
if $n\beta_{n}\to\infty$ then $\Pi_{n}^{\beta_{n}}\Rightarrow\delta_{\theta_{0}}$
$\text{P}$-a.s.w. (this verifies A6 with $\Pi_{0}=\delta_{\theta_{0}}$).

\paragraph{Verification of A5 (AUI).}

On the compact $\mathbb{T}$, 
\[
\sup_{\theta\in\mathbb{T}}\left|f_{n}(\theta)\right|\le\sup_{\gamma\in[s^{-1},s]}\left|\log(2\gamma)\right|+s\,n^{-1}\sum_{i=1}^{n}\left|X_{i}\right|+\sup_{(\mu,\gamma)\in\mathbb{T}}\left|\mu\right|/\gamma\text{.}
\]
By the strong law of large numbers and $\text{E}\left|X\right|<\infty$,
the right-hand side is $\text{P}$-a.s. eventually bounded by a deterministic
constant $M$. Hence, if $\delta>M$, then for all $n$ large enough
\[
\int_{\mathbb{T}}\left|f_{n}(\theta)\right|\mathbf{1}_{\{\left|f_{n}(\theta)\right|\ge\delta\}}\Pi_{n}^{\beta_{n}}(\text{d}\theta)=0\text{,}
\]
which establishes AUI.

\paragraph{Limits for BPIC and WBIC.}

By A2, $\left(f_{n}\right)_{n\in\mathbb{N}}$ converges continuously
to $f$, $\text{P-a.s.}$; by the above, A5 holds; and by Proposition
\ref{prop:-consistency-of-beta-posterior}, A6 holds with $\Pi_{0}=\delta_{\theta_{0}}$.
Therefore Theorem \ref{thm:-feinberg} applies and yields 
\[
\int_{\mathbb{T}}f_{n}(\theta)\,\Pi_{n}^{\beta_{n}}(\text{d}\theta)\xrightarrow[n\to\infty]{\text{P-a.s.}}f(\theta_{0})=-\log(2\gamma_{0})-\gamma_{0}^{-1}\text{E}\left|X-\mu_{0}\right|=-\log(2\gamma_{0})-1\text{.}
\]
Hence, as $n\to\infty$ with $n\beta_{n}\to\infty$, 
\[
\text{BPIC}_{n},\text{WBIC}_{n}\xrightarrow[n\to\infty]{\text{P-a.s.}}-2f(\theta_{0})=2\log(2\gamma_{0})+2\text{.}
\]

\paragraph{Limit for DIC.}

For $\beta_{n}=1$, A7 for $\left\Vert \cdot\right\Vert $ with respect
to $\left(\Pi_{n}\right)_{n\in\mathbb{N}}$ holds trivially since
$\mathbb{T}$ is compact. By Proposition \ref{prop:-mean-convergence},
together with A1, A2 and A6, we have $\bar{\theta}_{n}\xrightarrow[n\to\infty]{\text{P-a.s.}}\theta_{0}$
and therefore 
\[
\frac{2}{n}\sum_{i=1}^{n}\log\text{p}(X_{i}\mid\bar{\theta}_{n})\xrightarrow[n\to\infty]{\text{P-a.s.}}2\text{E}\left[\log\text{p}(X\mid\theta_{0})\right]\text{.}
\]
Combining this with the first term yields 
\[
\text{DIC}_{n}\xrightarrow[n\to\infty]{\text{P-a.s.}}-2\,f(\theta_{0})=2\log(2\gamma_{0})+2\text{.}
\]

\subsection{Approximation of \texorpdfstring{$\text{DIC}_{n}$}{DICn} for the geometric model}

Under the geometric model, we recall that the PMF of $X$ is 
\[
\text{p}\left(x\mid\theta\right)=\left(1-\theta\right)^{x}\theta\text{,}
\]
for $x\in\mathbb{X}=\mathbb{N}\cup\{0\}$ and $\theta\in\mathbb{T}=(0,1)$. We now endow $\theta$ with a prior $\Pi$ whose law is $\text{Beta}\left(\alpha,\beta\right)$, $\alpha,\beta>0$. The average log-likelihood has the form
\[
\frac{1}{n}\sum_{i=1}^{n}\log\text{p}\left(X_{i}\mid\theta\right)=\bar{X}_{n}\log\left(1-\theta\right)+\log\theta\text{.}
\]
To evaluate $\text{DIC}_{n}$, we must evaluate the two expressions
\[
\int_{0}^{1}\frac{1}{n}\sum_{i=1}^{n}\log\text{p}\left(X_{i}\mid\theta\right)\Pi_{n}\left(\text{d}\theta\right)\text{ and }\frac{1}{n}\sum_{i=1}^{n}\log\text{p}\left(X_{i}\mid\bar{\theta}_{n}\right)\text{,}
\]
where $\bar{\theta}_{n}$ is the expectation under $\Pi_{n}$. Conjugacy between the geometric model and the beta prior implies that
$\Pi_{n}$ has law $\text{Beta}\left(a_{n},b_{n}\right)$, where $a_{n}=n+\alpha$ and $b_{n}=n\bar{X}_{n}+\beta$. Furthermore, if $Y\sim\text{Beta}\left(a_{n},b_{n}\right)$, then
\[
\text{E}Y=\frac{a_{n}}{a_{n}+b_{n}}\text{,}\qquad \text{E}\log Y=\psi\left(a_{n}\right)-\psi\left(a_{n}+b_{n}\right)\text{,}\qquad \text{E}\log\left(1-Y\right)=\psi\left(b_{n}\right)-\psi\left(a_{n}+b_{n}\right)\text{.}
\]
Given these facts,
\begin{align*}
\int_{0}^{1}\frac{1}{n}\sum_{i=1}^{n}\log\text{p}\left(X_{i}\mid\theta\right)\Pi_{n}\left(\text{d}\theta\right)
&=\text{E}\left[\bar{X}_{n}\log\left(1-Y\right)+\log Y\,\big|\,\bar{X}_{n}\right]\\
&=\bar{X}_{n}\left\{\psi\left(b_{n}\right)-\psi\left(a_{n}+b_{n}\right)\right\}+\psi\left(a_{n}\right)-\psi\left(a_{n}+b_{n}\right)\text{.}
\end{align*}
We then apply the Poincar\'e-type approximation
\[
\psi\left(a\right)\approx\log a-\frac{1}{2a}
\]
to obtain
\[
\bar{X}_{n}\left\{\log\left(\frac{b_{n}}{a_{n}+b_{n}}\right)-\frac{a_{n}}{2b_{n}\left(a_{n}+b_{n}\right)}\right\}
+\left\{\log\left(\frac{a_{n}}{a_{n}+b_{n}}\right)-\frac{b_{n}}{2a_{n}\left(a_{n}+b_{n}\right)}\right\}\text{.}
\]
By substitution of $\bar{\theta}_{n}=\text{E}Y=a_{n}/(a_{n}+b_{n})$, we also have
\[
\frac{1}{n}\sum_{i=1}^{n}\log\text{p}\left(X_{i}\mid\bar{\theta}_{n}\right)
=\bar{X}_{n}\log\left(\frac{b_{n}}{a_{n}+b_{n}}\right)+\log\left(\frac{a_{n}}{a_{n}+b_{n}}\right)\text{.}
\]
The approximation for $\text{DIC}_{n}$ then follows by direct substitution into its definition.

\bibliographystyle{apalike2}
\bibliography{2023_MathsETC}

\end{document}